\documentclass[11pt,a4paper]{article}
\usepackage[T1]{fontenc}
\usepackage[utf8]{inputenc}
\usepackage[a4paper,margin=25mm]{geometry}
\usepackage{amsmath,amssymb,amsfonts,amsthm,mathrsfs}
\usepackage{graphicx,float,placeins}
\usepackage[dvipsnames]{xcolor}
\usepackage{microtype}
\usepackage[hidelinks]{hyperref}
\numberwithin{equation}{section}
\theoremstyle{plain}
\newtheorem{theorem}{Theorem}[section]
\newtheorem{lemma}[theorem]{Lemma}
\newtheorem{corollary}[theorem]{Corollary}
\newtheorem{proposition}[theorem]{Proposition}
\theoremstyle{definition}
\newtheorem{definition}[theorem]{Definition}
\newtheorem{assumption}[theorem]{Assumption}

\theoremstyle{remark}
\newtheorem{remark}[theorem]{Remark}
\newenvironment{keywords}{\par\smallskip\noindent\textbf{Keywords.}\ }{\par\smallskip}
\newenvironment{MSCcodes}{\par\noindent\textbf{2020 Mathematics Subject Classification.}\ }{\par\medskip}
\def\R{\mathbb{R}}
\def\1{{\bf 1}}

\def\d{\mathrm{d}}
\date{}
\title{Exponential Contraction for Underdamped Langevin Diffusions with Superlinear Forces}
\author{Shan Huang\thanks{School of Mathematics and Statistics, Lanzhou University, Lanzhou, Gansu 730000, China; School of Mathematics and Statistics, Northeast Normal University, Changchun, Jilin 130024, China.}
\and Xiaoyue Li\thanks{Corresponding author: lixy@tiangong.edu.cn. School of Mathematical Sciences, Tiangong University, Tianjin 300387, China. This author was supported by the National Natural Science Foundation of China (No. 12371402) and the Natural Science Foundation of Tianjin Municipality (No. 24JCZDJC00830). }
}
\begin{document}
\maketitle
\begin{abstract}
We establish exponential contraction for underdamped Langevin diffusions with nonconvex confining potentials and possibly superlinear forces. The force satisfies a polynomial local Lipschitz bound with exponent $\ell\geq0$, and the potential has a lower bound of order $m\geq2$, with $3\ell<2m$. Under radial confinement and a potential--dissipation inequality, we construct an explicit weighted Kantorovich transportation cost in which the Markov semigroup contracts probability laws exponentially. The estimate holds for all nonnegative times, with prefactor one and an explicit positive rate, without a smallness condition on the local Lipschitz coefficient. The proof combines localized reflection--synchronous coupling with an exponential Lyapunov weight. A radius-dependent shift of the position--velocity difference balances the local force contribution to the distance drift, while distinct localization scales in position and velocity exploit the growth of the potential. Together, these choices yield a dissipative estimate for the weighted cost with a vanishing regularization error. The regularized couplings preserve the Langevin marginals, so passage to weak limits yields contraction for all initial laws with finite transportation cost. Consequently, the dynamics converge exponentially to the Gibbs measure within the corresponding finite-cost class. The assumptions include quartic double-well potentials of arbitrary amplitude. Numerical experiments for asymmetric quartic wells examine transportation between two initial laws and quantify the conservatism of the explicit rate.

\end{abstract}
\section{Introduction}
Underdamped Langevin dynamics provide a stochastic model for damped Hamiltonian systems and a basis for sampling from Gibbs distributions \cite{ChengEtAl2018,DalalyanRiou2020}. In units where the mass and thermal energy are one, they satisfy
\begin{equation}\label{eq4.1}
	\mathrm dX_t=V_t\,\mathrm dt,\qquad
	\mathrm dV_t=-\gamma V_t\,\mathrm dt-\nabla U(X_t)\,\mathrm dt+\sqrt{2\gamma}\,\mathrm dB_t.
\end{equation}
Here, $X_t,V_t\in\mathbb R^d$ denote position and velocity, $\gamma>0$ is the friction parameter, and $B$ is a standard $d$-dimensional Brownian motion. The confining potential $U$, with $e^{-U}$ integrable, determines the Gibbs probability measure
\[
\mu_*(\mathrm dx,\mathrm dv)=\mathfrak Z^{-1}e^{-U(x)-|v|^2/2}\,\mathrm dx\,\mathrm dv,
\qquad \mathfrak Z=(2\pi)^{d/2}\int_{\mathbb R^d}e^{-U(x)}\,\mathrm dx.
\]
We study quantitative contraction between solution laws for nonconvex potentials whose gradients are locally Lipschitz and may grow superlinearly. Our aim is an explicit contraction estimate valid for all nonnegative times, without a smallness condition on the coefficient in the local Lipschitz bound.

Quantitative convergence to equilibrium has been studied extensively through the associated kinetic Fokker--Planck equation. Hypocoercivity estimates control exponential convergence of the evolving density to the Gibbs density in suitable weighted Sobolev and $L^2$ norms under confinement, regularity, and functional-inequality assumptions \cite{HerauNier2004,Villani2009,DMS2015}. More recent work gives explicit $L^2(\mu_*)$ decay estimates for centered observables under the Langevin semigroup, with explicit dependence on friction \cite{CaoLuWang2023,EberleLoerlerJFA2026}, and extends quantitative convergence estimates to nonquadratic kinetic energies \cite{BrigatiStoltz2025}. Lyapunov methods also yield exponential convergence of transition laws to an invariant measure for damped Hamiltonian systems with stochastic forcing under confinement conditions \cite{Wu2001}. For Langevin dynamics, Mattingly, Stuart and Higham \cite{MSH2002} obtain convergence in weighted total variation for smooth nonnegative potentials satisfying a radial potential--dissipation condition, including quartic double wells. Thus nonconvexity and superlinear forces already occur in classical ergodicity results. General Harris-type criteria give weighted total-variation contraction for a transition kernel under geometric drift and minorization \cite{HairerMattingly2011}, and weighted transportation contraction at a suitable positive time under geometric drift, local contraction, and a small-set condition in a distance-like cost \cite{HMS2011}. The question considered here is more specific: an explicit transportation contraction between evolving laws, with prefactor one for every $t\geq0$.

Reflection coupling gives a direct route to quantitative transportation contraction. For diffusions with constant nondegenerate noise, concave distance functions and Lyapunov weights accommodate nonconvexity under large-distance dissipativity \cite{Eberle2016} or a geometric Lyapunov condition and suitable bounds on the drift and the growth of the Lyapunov function \cite{EGZHarris2019}. For underdamped Langevin dynamics, however, the noise acts only on velocity, requiring a coupling adapted to the position--velocity structure. Bolley, Guillin and Malrieu \cite{BGM2010} use synchronous coupling to obtain contraction in a modified quadratic Wasserstein distance for a Vlasov--Fokker--Planck equation with a globally Lipschitz, strongly monotone friction map, a small Lipschitz perturbation of a linear confining force, and sufficiently weak odd Lipschitz interaction. For \eqref{eq4.1}, Eberle, Guillin and Zimmer \cite{EGZKinetic2019} combine reflection and synchronous coupling to obtain an explicit weighted transportation contraction under a global Lipschitz bound on the force and a radial potential--dissipation condition. Their estimate has prefactor one and holds for all nonnegative times. Further kinetic contraction results include contraction at a dimension-free rate in a transportation distance equivalent to $W_1$, under additional dissipativity and friction conditions \cite{Schuh2024}, and weighted Wasserstein estimates with a general prefactor obtained by PDE methods under nondegenerate kinetic transport and compatible Lyapunov conditions \cite{ForcilloPorretta2026}. These kinetic results impose global Lipschitz conditions, which exclude genuinely superlinear forces.

The extension to locally Lipschitz forces requires control of a state-dependent local Lipschitz bound. Eberle, Guillin and Zimmer \cite[Remark~2.5]{EGZKinetic2019} discuss this extension and emphasize the additional difficulty of tracking and optimizing the constants in the contraction estimate. Guillin, Le Bris and Monmarch\'e \cite{GBM2022} extend the reflection--synchronous coupling approach to Vlasov--Fokker--Planck equations with nonconvex confining potentials and sufficiently weak interactions. Under their confinement and initial-moment conditions, the $W_1$ and $W_2$ distances between two solution laws decay exponentially. Their analysis includes locally Lipschitz confining forces, subject to an explicit smallness condition on the coefficient of the state-dependent part of the local Lipschitz bound. This restriction remains in the noninteracting case, which recovers \eqref{eq4.1} with $\gamma=1$. These results motivate the following question. Can one retain explicit contraction for all nonnegative times for superlinear forces without requiring the coefficient in their local Lipschitz bound to be small?

We answer this question for a class of possibly nonconvex $C^1$ potentials satisfying
\[
|\nabla U(x)-\nabla U(y)|
\leq L_1(1+|x|^\ell+|y|^\ell)|x-y|,
\qquad
U(x)\geq b_0|x|^m-B_0,
\]
where $L_1,b_0>0$, $B_0\geq0$, $m\geq2$, and $0\leq\ell<2m/3$, together with the remaining conditions in Assumptions~\ref{as4.0}--\ref{as4.3}. These include $U\geq0$, radial confinement, and a potential--dissipation inequality. No smallness condition is imposed on $L_1$. Writing $P_t$ for the transition semigroup, we construct a continuous weighted cost $\rho_R$ such that
\[
\mathcal W_{\rho_R}(\mu P_t,\nu P_t)
\leq e^{-ct}\mathcal W_{\rho_R}(\mu,\nu),\qquad t\geq0,
\]
with an explicit rate $c>0$ (Theorem~\ref{Th}). The cost is fixed independently of time, and the estimate holds for every pair of initial laws with finite initial transportation cost. Since $\mu_*$ is invariant, this yields exponential convergence to $\mu_*$ in $\mathcal W_{\rho_R}$ whenever $\mathcal W_{\rho_R}(\mu,\mu_*)<\infty$. It also gives uniqueness of the invariant probability measure within this finite-cost class (Corollary~\ref{Co1}).

The contribution is the combination of explicit contraction, prefactor one, and the absence of a smallness restriction on $L_1$. For the class $3\ell<2m$, this addresses the quantitative extension from global to local Lipschitz bounds discussed in \cite[Remark~2.5]{EGZKinetic2019}. The growth condition uses the strength of confinement: the quadratic lower bound implied by radial confinement gives $\ell<4/3$, while a lower bound of order $m=\ell+2$ permits $\ell<4$. In particular, the assumptions cover quartic double-well potentials of arbitrary amplitude. This differs from the explicit smallness regime for locally Lipschitz forces in \cite{GBM2022}. Remark~\ref{rem:assumption-examples} gives a direct necessary condition showing that sufficiently large quartic amplitudes lie outside that regime in the same normalization. The same remark verifies a smooth strongly convex family with growth orders $2<p<6$. Section~\ref{Sec6} studies asymmetric quartic double wells, which are sampling benchmarks for exploring unequal probability modes \cite{LuSpiliopoulos2018}. It compares numerical transportation bounds between two initial laws with the explicit continuous-time estimate and evaluates the conservatism of its rate.

For $\ell>0$, the main difficulty is that small separation between two states does not give a uniform bound on the local Lipschitz factor when both positions are large. We therefore modify the switching regions in \cite{EGZKinetic2019,GBM2022} by localizing reflection coupling to pairs for which both states lie in
\(S_R=\{(x,v):|x|<R^{2/m},\ |v|<R\},\)
with a cutoff near the diagonal, and use synchronous coupling outside this region. The two scales match the lower bound $U(x)\geq b_0|x|^m-B_0$ and the quadratic kinetic energy. For $\xi=(x,v)$ and $\xi'=(x',v')$, set $z=x-x'$ and $y=v-v'$. With
\[
L_R=L_1(1+2R^{2\ell/m}),\qquad
\alpha_R=1+\frac{\sqrt{L_R}}{\gamma},\qquad
A_R=1+\frac{2\sqrt{L_R}}{\gamma},
\]
define
\(q=\alpha_Rz+\gamma^{-1}y,~ r=A_R|z|+|q|.\)
Allowing the shift $\alpha_R$ to depend on the local force bound makes both distance coefficients grow only as $\sqrt{L_R}$. For $\xi,\xi'\in\overline S_R$, the deterministic drift contribution to $r$ is bounded above by $-\gamma r+\gamma(A_R+\alpha_R)|q|$. Thus the negative term $-\gamma r$ remains available near $q=0$, where regularization weakens reflection, with an error that vanishes as $|q|\to0$. We combine this distance with a nondecreasing concave function $f_R$ and an exponential Lyapunov function $F$, defining
\(\rho_R(\xi,\xi')=f_R(r)\bigl(1+\varepsilon_RF(\xi)+\varepsilon_RF(\xi')\bigr).\)
The concavity of $f_R$ controls the force and covariation terms where reflection is active, while the Lyapunov inequality controls the drift of the weighted cost whenever at least one state lies outside $S_R$.

The quantitative step is to choose $R$ while accounting for its effect on the local Lipschitz bound, the concave function $f_R$, and the weight coefficient $\varepsilon_R$. The core bound \eqref{eta2} and the profile estimate in Lemma~\ref{le4.3} give
\[
\log\frac1{\varepsilon_R}
\leq C\bigl(1+\log R+R^{3\ell/m}+R^{1+\ell/m}\bigr)=o(R^2).
\]
Here, $C$ depends on the model parameters but not on $R$. Whenever at least one of $\xi,\xi'$ lies outside $S_R$, the potential lower bound and the exponential Lyapunov function give $F(\xi)\vee F(\xi')\geq e^{\bigl(b_0/4\wedge c_0\bigr)R^2-B_0/4}$. Together with Proposition~\ref{pr4.1} and \eqref{energy-tail-bound}, these bounds control the drift of the weighted cost outside $S_R$ for sufficiently large $R$, without a smallness condition on $L_1$. The restriction $3\ell<2m$ ensures that both $3\ell/m$ and $1+\ell/m$ are strictly smaller than two. Finally, covariance compensation preserves the Langevin marginal path laws of the regularized couplings. Following the approximation and weak-limit strategy of \cite{EberleZimmer2019}, we pass the resulting bound on the expected coupling cost to weak limits of the coupled path laws as the regularization error vanishes, and obtain contraction.

Section~\ref{Sec2} states the assumptions and main results and defines the cost and quantitative parameters. Sections~\ref{Sec3} and \ref{Sec4} develop the regularized couplings and the required drift estimates. Section~\ref{Sec5} proves contraction by passage to weak limits and deduces convergence to the Gibbs measure. Section~\ref{Sec6} presents asymmetric quartic examples and numerical contraction diagnostics.

\section{Assumptions and main results}\label{Sec2}
This section fixes the notation and assumptions, introduces the quantitative constants and transportation cost, and states the contraction and ergodicity results.
\subsection{Notation and assumptions}
Throughout, write $\mathbb N=\{0,1,\ldots\}$ and $\mathbb R_+:=(0,\infty)$. For $a,b\in\mathbb R$, write $a\wedge b:=\min\{a,b\}$, $a\vee b:=\max\{a,b\}$, and $a^+:=a\vee0$. We use $|\cdot|$ and $\langle\cdot,\cdot\rangle$ for the Euclidean norm and inner product, $\|\cdot\|_{\rm op}$ for the matrix operator norm, and ${\mathrm T}$ for transpose. The \(d\times d\) identity matrix is denoted by \(\mathbf I_d\). Set
\[\mathbb S_d^+:=\{A\in\mathbb R^{d\times d}:A^{\mathrm T}=A,\ y^{\mathrm T}Ay\geq0\ \text{for every }y\in\mathbb R^d\}.\]
For symmetric matrices, $A\preceq B$ means $B-A\in\mathbb S_d^+$.
The symbol $C$
denotes a finite positive constant that may change from line to line. For a topological space $S$, $C_b(S)$ denotes the bounded continuous real-valued functions on $S$. We write $\1_A$ for an indicator, $\delta_\xi$ for the Dirac measure at $\xi$, and $\mu(g)=\int g\,\mathrm d\mu$ whenever the integral is defined. 

For $p\in\mathbb{N}$, a function $f:\mathbb{R}^d\to\mathbb{R}$ is said to belong to
$C^p(\mathbb{R}^d;\mathbb{R})$ if, for every multi-index
$\alpha=(\alpha_1,\ldots,\alpha_d)\in\mathbb{N}^d$ with \(|\alpha|:=\alpha_1+\cdots+\alpha_d\le p\),
the partial derivative \(D^\alpha f(x)\)
exists for all $x\in\mathbb{R}^d$ and the map $x\mapsto D^\alpha f(x)$ is continuous on
$\mathbb{R}^d$. For $f\in C^1(\mathbb{R}^d;\mathbb{R})$, its gradient is denoted by
$\nabla f$.
For $f\in C^2(\mathbb{R}^d;\mathbb{R})$, its Laplacian is denoted by $\Delta f$.
Let \((\Omega,\mathcal{F},\mathbb{P},\{\mathcal{F}_t\}_{t\ge0})\)
be a complete filtered probability space satisfying the usual conditions, namely,
$\{\mathcal{F}_t\}_{t\ge0}$ is right-continuous and $\mathcal{F}_0$ contains all
$\mathbb{P}$-null sets. Let $\{B_t\}_{t\ge0}$ be a $d$-dimensional
$\{\mathcal{F}_t\}_{t\ge0}$-Brownian motion.
Set \(E:=\mathbb R^d\times\mathbb R^d\). Let $\mathcal{P}(E)$ denote the space of probability measures on \(E\). For $\mu,\nu\in\mathcal{P}(E)$ and a nonnegative Borel cost $\varrho$ on $E\times E$, define
\begin{equation*}
\mathcal{W}_{\varrho}(\mu,\nu)
:=
\inf_{\Gamma\in\Pi(\mu,\nu)}
\int_{E\times E}
\varrho(\xi,\xi')\,\Gamma(\mathrm{d}\xi,\mathrm{d}\xi'),
\end{equation*}
where $\xi=(x,v)$ and $\xi'=(x',v')$. Here $\Pi(\mu,\nu)$ denotes the set of all couplings of $\mu$ and $\nu$. 

Let $(X_t,V_t)_{t\ge0}$ be the solution to \eqref{eq4.1}. For a bounded or nonnegative Borel measurable function $g:\mathbb{R}^d\times\mathbb{R}^d\to\mathbb{R}$, the transition semigroup of \eqref{eq4.1} is
\((P_tg)(x,v) := \mathbb{E}^{(x,v)}\big[g(X_t,V_t)\big],~ t\geq 0,\)
where $\mathbb{E}^{(x,v)}$ denotes expectation from
the initial state $(X_0,V_0)=(x,v)$. For $\mu\in\mathcal P(E)$, $\mu P_t$ is the law at time $t$ from initial law $\mu$, so $(\mu P_t)g=\mu(P_tg)$.
For \(h\in C^2(\mathbb R^{2d};\mathbb R)\), the generator of \eqref{eq4.1} is
\begin{equation}\label{eq4.3}
\mathcal Lh=\gamma\Delta_vh-\gamma\langle v,\nabla_vh\rangle-\langle\nabla U,\nabla_vh\rangle+\langle v,\nabla_xh\rangle.
\end{equation}

We impose the following assumptions on the potential.
\begin{assumption}\label{as4.0}
The potential \(U\) belongs to \(C^1(\mathbb R^d;\mathbb R)\). There exist constants \(m\ge2\), \(b_0>0\), and \(B_0\ge0\) such that
\begin{equation}\label{eq:power-lower}
 U(x)\ge b_0|x|^m-B_0,\qquad x\in\mathbb R^d.
\end{equation}
\end{assumption}
\begin{assumption}\label{as4.1}
There exist constants \(L_1>0\) and \(0\le\ell<2m/3\), with \(m\) from Assumption~\ref{as4.0}, such that
\[
|\nabla U(x)-\nabla U(y)|
\le L_1(1+|x|^\ell+|y|^\ell)|x-y|,
\qquad x,y\in\mathbb R^d.
\]
When \(\ell=0\), we use \(r^0=1\) for every \(r\ge0\), including \(r=0\).
\end{assumption}
\begin{remark}\label{re4.1}
Set \(C_1:=|\nabla U(\mathbf0)|+3L_1\). Assumption \ref{as4.1} implies
\[
|\nabla U(x)|
\le|\nabla U(\mathbf0)|+L_1(2+|x|^\ell)|x|
\le C_1(1+|x|^{\ell+1}).
\]
\end{remark}
\begin{assumption}\label{as4.2}
The potential satisfies \(U\ge0\). There exist constants \(K_1>0\) and
\(L_2>\gamma^2/4\) such that
\[\langle\nabla U(x),x\rangle\ge L_2|x|^2,\qquad |x|\ge K_1.\]
\end{assumption}
\begin{assumption}\label{as4.3}
There exist constants \(L_3>0\) and \(K_2>0\) such that
\[\langle\nabla U(x),x\rangle\ge L_3U(x)-K_2,\qquad x\in\mathbb R^d.\]
\end{assumption}
The lower bound \eqref{eq:power-lower} implies \(U(x)\to\infty\) as \(|x|\to\infty\) and \(e^{-U(x)}\le e^{B_0-b_0|x|^m}\), so the Gibbs normalization is finite. Assumption~\ref{as4.2} also gives a quadratic lower bound: integrating the radial inequality from radius \(K_1\) and using \(U\ge0\) yields
\begin{equation}\label{eq:quadratic-lower}
U(x)\ge\frac{L_2}{2}(|x|^2-K_1^2),\qquad x\in\mathbb R^d.
\end{equation}
Proposition~\ref{prop:gibbs} recalls Gibbs invariance at the stated regularity and verifies the \(F\)-moment used in Corollary~\ref{Co1}.

\begin{remark}\label{rem:assumption-examples}
The assumptions allow both smooth convex potentials and nonconvex potentials
with genuinely superlinear forces. For a smooth family, fix $\gamma=1$,
$\kappa>0$, and $2<p<6$, and let
\begin{equation}\label{eq:smooth-power-potential}
 U_\kappa(x)=\frac{|x|^2}{2}
 +\frac{\kappa}{p}\bigl((1+|x|^2)^{p/2}-1\bigr),
 \qquad x\in\R^d.
\end{equation}
Since $(1+s)^{p/2}\geq1+s^{p/2}$ for $s\geq0$, this potential satisfies
$U_\kappa(x)\geq |x|^2/2+(\kappa/p)|x|^p$.
Differentiating \eqref{eq:smooth-power-potential}, with $r=|x|$, gives
\begin{equation}\label{eq:smooth-power-hessian}
 \|\nabla^2U_\kappa(x)\|_{\rm op}
 \leq1+\kappa(p-1)(1+r^2)^{(p-2)/2}
 \leq L_1(1+r^{p-2}),
\end{equation}
where \(L_1=1+\kappa(p-1)2^{\max\{(p-4)/2,0\}}\). Here \((1+s)^a\le2^{(a-1)^+}(1+s^a)\) for \(a,s\ge0\) gives the second bound.
Every point on the segment from \(x\) to \(y\) has norm at most \(|x|\vee|y|\).
Integrating \eqref{eq:smooth-power-hessian} along that segment therefore gives Assumption~\ref{as4.1} with \(\ell=p-2\).
Moreover, $\langle x,\nabla U_\kappa(x)\rangle\geq(1+\kappa)|x|^2$.
The inequality
\[
 (1+s)^{p/2}-1\leq\frac p2s(1+s)^{(p-2)/2},\qquad s\geq0,
\]
gives $\langle x,\nabla U_\kappa(x)\rangle\geq2U_\kappa(x)$.
Thus all the assumptions hold for every $\kappa>0$, with
\[
\begin{gathered}
 m=p,\quad b_0=\frac{\kappa}{p},\quad B_0=0,\quad
 \ell=p-2,\quad L_2=1+\kappa,\quad K_1=1,\quad L_3=2,\quad K_2=1.
\end{gathered}
\]
The condition $3\ell<2m$ is exactly $p<6$ for this family.

A direct comparison with the explicit smallness regime of \cite{GBM2022}
is provided by the one-dimensional quartic double wells
\begin{equation}\label{eq:quartic-double-well}
 U_\kappa^{\rm dw}(x)=\frac{\kappa}{4}(x^2-1)^2,\qquad \kappa>0,
\end{equation}
which satisfy our assumptions for every $\kappa>0$, as verified in
Section~\ref{Sec6}. Consider zero interaction and $\gamma=1$, as in the
normalization of that reference. Assumption~1.4 and Remark~1.5 of
\cite{GBM2022} imply, for any admissible shift
$\widetilde U=U_\kappa^{\rm dw}+C$, $C\geq0$,
\(|\widetilde U''(x)| \leq L_U+2L_\psi\sqrt{\lambda x^2+24\widetilde U(x)},\)
where $L_U,\lambda>0$ are finite constants, while \cite[(5.9)]{GBM2022}
requires $L_\psi\leq1$.
Dividing the preceding inequality by $x^2$ and letting $|x|\to\infty$
gives
\[
 3\kappa\leq2L_\psi\sqrt{6\kappa},
 \qquad\hbox{hence}\qquad L_\psi\geq\sqrt{3\kappa/8}.
\]
Consequently, $\kappa>8/3$ is incompatible with that explicit smallness
condition, independently of $C$ and the remaining auxiliary parameters.
This comparison uses the same normalization of the dynamics.
\end{remark}

\subsection{Quantitative framework}
We combine an exponential Lyapunov weight with a two-scale concavity profile. The position coefficient in the kinetic distance provides negative deterministic drift when the shifted difference is small. The parameters are chosen in the order below.
First set
\begin{equation}\label{a}
\begin{aligned}
\theta&:=\frac12\left(1-\frac{\gamma^2}{4L_2}\right),\qquad\lambda_\theta:=\min\left\{(1-\theta)L_2-\frac{\gamma^2}{4},
1-\frac{\gamma^2}{4(1-\theta)L_2}\right\},\\
a&:=\frac{\gamma\min\{4\theta L_3,2\lambda_\theta,1\}}
{8\max\{\gamma^2,8\}},\qquad
c_0:=\frac18\min\left\{\frac{\gamma^2}{6},\frac14\right\}.
\end{aligned}
\end{equation}
These constants are strictly positive. For \(s>0\) and \(t\ge0\), let
\[
\Psi(s,t):=\begin{cases}
0,&t\le s,\\
t-s,&s<t<2s,\\
s\exp(t/s-2),&t\ge2s.
\end{cases}
\]
With \(C_1\) from Remark \ref{re4.1}, define
\begin{equation}\label{global-C-actual}
\begin{aligned}
C_2&:=(1-\theta)K_1C_1(1+K_1^{\ell+1})
+\left(\frac{3\gamma^2}{8}+\frac14\right)K_1^2+2d,\\
C_3&:=\frac14\min\{4\theta L_3,2\lambda_\theta,1\}+C_2+\theta K_2,\quad C_4:=\frac{\gamma C_3}{8},\quad C_5:=a\vee\Psi(a,C_4).
\end{aligned}
\end{equation}
The admissible regime is
\begin{equation}\label{growth-regime}
 m\ge2,\qquad 0\le\ell<2m/3,\qquad L_1>0\ \hbox{arbitrary}.
\end{equation}
Define
\begin{equation}\label{lyapunov-functions}
\begin{aligned}
\mathcal V(x,v)&:=\frac18
\left(2U(x)+\frac{\gamma^2}{2}|x|^2+\gamma\langle x,v\rangle+|v|^2\right),\\
F(x,v)&:=e^{\mathcal V(x,v)},\qquad G(x,v):=a(1+\mathcal V(x,v)),
\end{aligned}
\end{equation}
and
\begin{equation}\label{radius-1}
\begin{aligned}
\mathcal H:=\{(\xi,\xi')\in E^2:
F(\xi)G(\xi)+F(\xi')G(\xi')\le2C_5\},\quad R_1:=\sqrt{\frac{1}{c_0}\log\left(\frac{2C_5}{a}\right)}.
\end{aligned}
\end{equation}
For a provisional \(R\ge1\), set
\begin{equation}\label{SRset}
S_R:=\{(x,v)\in E:|x|<R^{2/m},\ |v|<R\},
\end{equation}
and
\begin{equation}\label{eq4.5}
L_R:=L_1(1+2R^{2\ell/m}),\qquad
\alpha_R:=1+\frac{\sqrt{L_R}}{\gamma},\qquad A_R:=1+\frac{2\sqrt{L_R}}{\gamma}.
\end{equation}
For \(z=x-x'\) and \(y=v-v'\), define the shifted difference \(q:=\alpha_Rz+\gamma^{-1}y\). The kinetic distance is
\begin{equation}\label{rxv}
r((x,v),(x',v')):=A_R|x-x'|
+|\alpha_R(x-x')+\gamma^{-1}(v-v')|.
\end{equation}
The notation \(r\) suppresses its dependence on the subsequently fixed \(R\).
\begin{remark}\label{rem211}
Define \(\mathbf e(w)=w/|w|\) for \(w\ne0\), and \(\mathbf e(0)=0\).
For \(\xi,\xi'\in\overline S_R\), Assumption~\ref{as4.1} and \eqref{SRset} give
\(|\nabla U(x)-\nabla U(x')|\le L_R|z|\).
The coefficients in \eqref{eq4.5} satisfy
\[
 -\gamma\alpha_RA_R+\gamma\alpha_R(\alpha_R-1)+L_R/\gamma=-\gamma A_R.
\]
Consequently, Cauchy--Schwarz gives the norm-drift bound
\begin{equation}\label{local-kinetic-drift}
\begin{aligned}
&-\gamma\alpha_RA_R|z|+\gamma\bigl(A_R+\alpha_R-1\bigr)|q|
-\gamma\alpha_R(\alpha_R-1)\mathbf e(q)\cdot z
-\gamma^{-1}\mathbf e(q)\cdot(\nabla U(x)-\nabla U(x'))\\
\le&-\gamma A_R|z|+\gamma\bigl(A_R+\alpha_R-1\bigr)|q|
=-\gamma r+\gamma(A_R+\alpha_R)|q|
\le\gamma\bigl(A_R+\alpha_R-1\bigr)r.
\end{aligned}
\end{equation}
The negative term \(-\gamma r\) controls the region where \(q\) is small; the coefficients \(A_R,\alpha_R,A_R+\alpha_R-1\) grow only as the square root of the local force modulus.
\end{remark}
Define the local scales and profile by
\begin{equation}\label{eq4.7}
\eta_{1,R}:=\frac{\gamma\vee2}{2}R,\qquad
\eta_{2,R}:=2R_1(A_R+\alpha_R+\gamma^{-1}),
\end{equation}
\begin{equation}\label{Seeq1}
h_R(s):=\frac{\bigl(A_R+\alpha_R-1\bigr)\gamma^2}{8}s^2+\eta_{1,R}\gamma s,\quad
\phi_R(s):=e^{-h_R(s)},\quad
\Phi_R(s):=\int_0^s\phi_R(u)\,\mathrm du,
\end{equation}
\begin{equation}\label{J1-definition}
J_{1,R}:=\gamma\int_0^{\eta_{2,R}}
\frac{\Phi_R(s)}{\phi_R(s)}\,\mathrm ds,
\end{equation}
\begin{equation}\label{eq4.8}
\eta_{3,R}:=J_{1,R}^{-1}\wedge\frac{\gamma}{32}\phi_R(\eta_{2,R}).
\end{equation}
In particular, the core scale \(\eta_{2,R}\) depends on \(R\).
To choose \(R\) without circularity, introduce the following majorants,
which depend only on the original model parameters:
\begin{equation}\label{radius-majorants}
\begin{aligned}
\widehat A_1&:=2R_1\left(2+\frac{3\sqrt{3L_1}}{\gamma}+\frac1\gamma\right),\qquad \widehat A_2:=\frac{\gamma^2}{8}\left(1+\frac{3\sqrt{3L_1}}{\gamma}\right)\widehat A_1^2,
\qquad \widehat A_3:=\frac{\gamma(\gamma\vee2)}2\widehat A_1,\\
C_6&:=\frac{\gamma\widehat A_1^2}{2}\vee\frac{32}{\gamma},
\qquad C_7:=\log(1+2C_5C_6).
\end{aligned}
\end{equation}
The sufficient radii are
\begin{equation}\label{radius-2}
 R_2:=1\vee\sqrt{\frac{B_0}{2\bigl(b_0/4\wedge c_0\bigr)}}
 \vee\sqrt{\frac{24(\gamma+1)}{a\bigl(b_0/4\wedge c_0\bigr)}}\vee\left[\frac{24}{a\bigl(b_0/4\wedge c_0\bigr)}
 \left(3\sqrt{3L_1}+\frac{3L_1}{\gamma}\right)\right]^{\frac{m}{2(m-\ell)}}.
\end{equation}
\begin{equation}\label{radius-3}
\begin{aligned}
 R_3:=1\vee\sqrt{\frac{B_0+4C_7}{\bigl(b_0/4\wedge c_0\bigr)}}\vee\frac{8\ell}{m\bigl(b_0/4\wedge c_0\bigr)}\vee\left(\frac{4\widehat A_2}{\bigl(b_0/4\wedge c_0\bigr)}\right)^{\frac{m}{2m-3\ell}}
 \vee\left(\frac{4\widehat A_3}{\bigl(b_0/4\wedge c_0\bigr)}\right)^{\frac{m}{m-\ell}},
\end{aligned}
\end{equation}
\begin{equation}\label{radius-45}
 R_4:=\sqrt{\frac{B_0/4+\log(4C_5/a)}{\bigl(b_0/4\wedge c_0\bigr)}},
\end{equation}
and
\begin{equation}\label{Rstar-explicit}
 R_\star:=1+\max\{R_1,R_1^{m/2},R_2,R_3,R_4\}.
\end{equation}
Every quantity on the right is independent of the provisional radius.
Fix any \(R\ge R_\star\), and evaluate \eqref{SRset}--\eqref{eq4.8}
at this \(R\). All radius-dependent parameters remain fixed in the coupling
and limiting arguments. Finally, put
\begin{equation}\label{outer-profile-constants}
\eta_{4,R}:=2(A_R+\alpha_R)R^{2/m}+2\gamma^{-1}R,\quad J_{2,R}:=\gamma\int_0^{\eta_{4,R}}\frac{\Phi_R(s)}{\phi_R(s)}\,\mathrm ds,
\quad \eta_{5,R}:=J_{2,R}^{-1}\wedge\frac{\gamma}{32}\phi_R(\eta_{4,R}).
\end{equation}
Since \eqref{Rstar-explicit} gives \(R>R_1\) and \(R^{2/m}>R_1\), the definitions in \eqref{eq4.7} give \(\eta_{2,R}<\eta_{4,R}\). Both integrals \(J_{1,R},J_{2,R}\) are finite and
strictly positive.
Define
\begin{equation}\label{Seeq3}
\begin{aligned}
g_R(r):={}&1-\frac{\eta_{5,R}\gamma}{4}
\int_0^{r\wedge\eta_{4,R}}\frac{\Phi_R(s)}{\phi_R(s)}\,\mathrm ds-\frac{\eta_{3,R}\gamma}{4}
\int_0^{r\wedge\eta_{2,R}}\frac{\Phi_R(s)}{\phi_R(s)}\,\mathrm ds,
\end{aligned}
\end{equation}
and
\begin{equation*}
f_R(r):=\int_0^{r\wedge\eta_{4,R}}\phi_R(s)g_R(s)\,\mathrm ds,\qquad r\ge0.
\end{equation*}
The first subtracted term in \eqref{Seeq3} is bounded by
\(\eta_{5,R}J_{2,R}/4\le1/4\), and the second by
\(\eta_{3,R}J_{1,R}/4\le1/4\). Thus \(1/2\le g_R\le1\).
Since \(\phi_R\) is decreasing,
\begin{equation*}
r\phi_R(\eta_{4,R})\le\Phi_R(r)\le2f_R(r)
\le2\Phi_R(r)\le2r,\qquad 0\le r\le\eta_{4,R}.
\end{equation*}
The function \(f_R\) is continuously differentiable on
\((0,\eta_{4,R})\), constant on \([\eta_{4,R},\infty)\), and twice
continuously differentiable on
\((0,\eta_{2,R})\cup(\eta_{2,R},\eta_{4,R})\). On these two intervals,
\begin{equation}\label{fe}
f_R''(r)=-\left(\frac{\bigl(A_R+\alpha_R-1\bigr)\gamma^2}{4}r+\eta_{1,R}\gamma\right)f_R'(r)
-\frac{\eta_{5,R}\gamma}{4}\Phi_R(r)
-\frac{\eta_{3,R}\gamma}{4}\Phi_R(r)\1_{(0,\eta_{2,R})}(r).
\end{equation}
The right-hand side is nonpositive. The derivative is continuous
at \(\eta_{2,R}\) and jumps down to zero at \(\eta_{4,R}\).
Therefore \(f_R\) is nondecreasing and concave on \([0,\infty)\).

Let \(f'_{R,-}\) denote the left derivative. Extend \(f_R\) by \(f_R(r)=r\)
for \(r<0\). Since \(f_R'(0+)=1\), this extension is continuously
differentiable at zero. Its distributional second derivative \(\mu_{f_R}\)
is the nonpositive measure characterized by \(\mu_{f_R}([r,s))=f'_{R,-}(s)-f'_{R,-}(r)\) for \(r<s\).
It agrees with \(f_R''(r)\,\mathrm dr\) away from the two cutoffs and
vanishes on \((-\infty,0]\cup(\eta_{4,R},\infty)\). At the cutoffs,
\begin{equation}\label{R1}
\mu_{f_R}(\{\eta_{2,R}\})=0,\qquad \mu_{f_R}(\{\eta_{4,R}\})=-f'_{R,-}(\eta_{4,R})<0.
\end{equation}
We set the classical \(f_R''\) to zero at \(0,\eta_{2,R},\eta_{4,R}\)
and on \((\eta_{4,R},\infty)\).
Finally, set
\begin{equation}\label{epsilon}
\varepsilon_R:=\frac{\eta_{3,R}}{2C_5},
\end{equation}
and define
\begin{equation}\label{rho}
\rho_R((x,v),(x',v')):=
f_R(r((x,v),(x',v')))
\bigl(1+\varepsilon_RF(x,v)+\varepsilon_RF(x',v')\bigr).
\end{equation}
\begin{remark}\label{rho-remark}
The cost \(\rho_R\) is finite, continuous, symmetric, nonnegative, and
vanishes only on the diagonal. The proof does not require a triangle
inequality. Whenever \(R\) is fixed, we write
\(h=h_R\), \(\phi=\phi_R\), \(\Phi=\Phi_R\), \(g=g_R\), \(f=f_R\),
and \(\mu_f=\mu_{f_R}\).
\end{remark}

\subsection{Exponential ergodicity}
\begin{theorem}\label{Th}
Suppose that Assumptions \ref{as4.0}--\ref{as4.3} hold.
Fix \(R\ge R_\star\), and set \(c:=1\wedge\eta_{5,R}>0\).
Then, for every \(\mu,\nu\in\mathcal P(E)\) and \(t\ge0\),
\[\mathcal W_{\rho_R}(\mu P_t,\nu P_t) \le e^{-ct}\mathcal W_{\rho_R}(\mu,\nu).\]
The inequality is understood in the extended nonnegative reals.
Finite initial transportation cost suffices; finite \(F\)-moments of
both marginals are a sufficient, but not necessary, condition.
\end{theorem}
\begin{remark}\label{re4.3}
The geometric coefficients, the core cutoff, and the concavity profile depend on \(R\).
Lemma~\ref{le4.3} estimates this dependence through
\(h_R(\eta_{2,R})\le\widehat A_2R^{3\ell/m}+\widehat A_3R^{1+\ell/m}.\)
Both powers are smaller than two by \eqref{growth-regime}.
Equation~\eqref{profile-tail-balance} then compares the reciprocal weight with the exterior energy bound \eqref{energy-tail-bound}.
If only the quadratic lower bound \eqref{eq:quadratic-lower} is used, one may choose \(m=2\) and obtain \(\ell<4/3\). A quartic lower bound permits \(\ell=2\); the examples in Remark~\ref{rem:assumption-examples} and Section~\ref{Sec6} verify concrete parameter choices.
\end{remark}
\begin{remark}
The definition of \(J_{2,R}\) in \eqref{outer-profile-constants}, \(\Phi_R(s)\le s\), and the monotonicity of \(h_R\) give
\(J_{2,R}\le\gamma\eta_{4,R}^2e^{h_R(\eta_{4,R})}/2.\)
Therefore,
\[
 \eta_{5,R}\ge
 \left(\frac{2}{\gamma\eta_{4,R}^2}\wedge\frac{\gamma}{32}\right)
 \exp\left\{-\frac{\bigl(A_R+\alpha_R-1\bigr)\gamma^2}{8}\eta_{4,R}^2
 -\eta_{1,R}\gamma\eta_{4,R}\right\}.
\]
By \eqref{eq4.5}, \eqref{radius-majorants}, and \(R^{2/m}\le R\),
\(\eta_{4,R}\le(\widehat A_1/R_1)R^{1+\ell/m}\) and
\(\bigl(A_R+\alpha_R-1\bigr)\le(1+3\sqrt{3L_1}/\gamma)R^{\ell/m}\).
Together with \(\eta_{1,R}=(\gamma\vee2)R/2\), these bounds give
\[
 \eta_{5,R}\ge c_1R^{-2-2\ell/m}\exp\{-c_2R^{2+3\ell/m}\},\qquad R\ge R_\star,
\]
where one may take
\[
\begin{aligned}
	c_1&:=\frac{\gamma}{
		2\bigl[4\vee(2\gamma+3\sqrt{3L_1}+1)\bigr]^2},\\
	c_2&:=(2\gamma+3\sqrt{3L_1}+1)
	\left[
	(\gamma\vee2)
	+\frac{(\gamma+3\sqrt{3L_1})
		(2\gamma+3\sqrt{3L_1}+1)}{2\gamma}
	\right].
\end{aligned}
\]
\end{remark}
\begin{corollary}\label{Co1}
Under the assumptions of Theorem \ref{Th}, the Gibbs probability
\(\mu_*\) satisfies \(\mu_*P_t=\mu_*\) and \(\mu_*F<\infty\).
For every \(\mu\in\mathcal P(E)\) with
\(\mathcal W_{\rho_R}(\mu,\mu_*)<\infty\),
\[
\mathcal W_{\rho_R}(\mu P_t,\mu_*)
\le e^{-ct}\mathcal W_{\rho_R}(\mu,\mu_*),\qquad t\ge0.
\]
Moreover, \(\mu_*\) is the unique invariant probability in \(\{\nu\in\mathcal P(E):\mathcal W_{\rho_R}(\nu,\mu_*)<\infty\}\).
\end{corollary}
Gibbs invariance and Corollary \ref{Co1} are proved in Section \ref{Sec5}.

\section{Coupling construction}\label{Sec3}
Fix \(R\ge R_\star\), so that the geometric coefficients in Section~\ref{Sec2} remain constant throughout this section.
We construct regularized reflection--synchronous couplings with the exact
Langevin marginals and identify the properties retained by their weak limits.
The switching functions chosen at the end of the section will be used in
the localized estimates of Section~\ref{Sec4}.

Let \(\Xi_1:=C([0,\infty);E)\), endowed with uniform convergence on compact
time intervals and its Borel \(\sigma\)-field. Under
Assumptions~\ref{as4.0}--\ref{as4.3}, equation~\eqref{eq4.1} has a unique
non-explosive strong solution for every deterministic initial state, and
is unique in law. The local construction and energy estimate establishing
these facts are given in the proof of Theorem~\ref{thwcc} below.
We denote the corresponding path law by
\(\mathbf P^{(x,v)} :=\operatorname{Law}\bigl((X_t,V_t)_{t\ge0}\bigr)\in\mathcal P(\Xi_1).\)
Weak solutions are understood in the sense of
\cite[Chapter~5, Definition~3.1, p.~300]{KS1991}.

\begin{definition}\label{wcoup} For \((x,v),(x',v')\in E\), a weak coupling of solutions to \eqref{eq4.1} is a probability measure
\(\mathbf Q\in\Pi(\mathbf P^{(x,v)},\mathbf P^{(x',v')})\) on \(\Xi_1\times\Xi_1\). On the canonical probability space \((\Xi_1\times\Xi_1,
\mathcal{B}(\Xi_1)\otimes\mathcal{B}(\Xi_1),
\mathbf{Q})\), define the two canonical coordinate processes by
\[(X_t,V_t)(\omega,\omega'):=\omega(t),\qquad (X'_t,V'_t)(\omega,\omega'):=\omega'(t),\qquad t\geq 0.\]
The marginal conditions imply that
\[\begin{aligned}
\operatorname{Law}_{\mathbf{Q}}\bigl((X_t,V_t)_{t\geq 0}\bigr)=\mathbf{P}^{(x, v)}\qquad
\text{and}\qquad \operatorname{Law}_{\mathbf{Q}}\bigl((X'_t,V'_t)_{t\geq 0}\bigr)=\mathbf{P}^{(x', v')}.
\end{aligned}\]
In particular, \((X_0,V_0)=(x, v)\), \((X'_0,V'_0)=(x', v')\), \(\mathbf{Q}\text{-a.s.}\).
Each coordinate, with the usual augmentation of its own natural filtration, is a weak solution to \eqref{eq4.1} with the prescribed initial state.
\end{definition}
\begin{remark}\label{rewc}
A weak coupling specifies the two path marginals; it does not require
Brownian drivers for both equations relative to the joint canonical
filtration. Its time marginals nevertheless satisfy
\[
\mathbf Q\circ\bigl((X_t,V_t),(X'_t,V'_t)\bigr)^{-1}
\in\Pi\bigl(\delta_{(x,v)}P_t,\delta_{(x',v')}P_t\bigr).
\]
Thus, for every nonnegative Borel cost \(\varrho\), whenever the
right-hand side is finite,
\[
\mathcal W_\varrho\bigl(\delta_{(x,v)}P_t,\delta_{(x',v')}P_t\bigr)
\le\mathbb E_{\mathbf Q}
\varrho\bigl((X_t,V_t),(X'_t,V'_t)\bigr).
\]
\end{remark}

Remark \ref{rewc} shows that, for Wasserstein estimates, it is enough to construct a coupling at the level of path-space laws. For \(n\ge2\), choose \(\chi_n\in C^\infty([0,\infty);[0,1])\) such that
\(\chi_n(r)=0\) for \(r\le n^{-1}\) and \(\chi_n(r)=1\) for
\(r\ge2n^{-1}\). Using \(\mathbf e\) from Remark~\ref{rem211}, set
\(\mathbf e_n(q):=\chi_n(|q|)\mathbf e(q)\).
Fix \(\beta\in(0,2)\), put \(\theta_{n,\beta}:=1-n^{-\beta}\), and define
\[
\mathsf R_{n,\beta}(q)
:=\mathbf I_d-2\theta_{n,\beta}\mathbf e_n(q)\mathbf e_n(q)^{\mathrm T},
\qquad
\Sigma_{n,\beta}(q)
:=\bigl(\mathbf I_d-\mathsf R_{n,\beta}(q)
\mathsf R_{n,\beta}(q)^{\mathrm T}\bigr)^{1/2}.
\]
The square root is the symmetric positive semidefinite one. Indeed, the
rank-one structure gives the explicit formula
\[
\Sigma_{n,\beta}(q)
=2\sqrt{\theta_{n,\beta}}\,\chi_n(|q|)
\sqrt{1-\theta_{n,\beta}\chi_n(|q|)^2}\,
\mathbf e(q)\mathbf e(q)^{\mathrm T}.
\]
Both matrix maps are bounded and locally Lipschitz: the cutoff removes
the singularity at \(q=0\), and
\(1-\theta_{n,\beta}\chi_n^2\ge1-\theta_{n,\beta}>0\) for fixed \(n\). The compensation identity is
\begin{equation}\label{AAI}
\mathsf R_{n,\beta}(q)\mathsf R_{n,\beta}(q)^{\mathrm T}
+\Sigma_{n,\beta}(q)\Sigma_{n,\beta}(q)^{\mathrm T}
=\mathbf I_d,\qquad q\in\mathbb R^d.
\end{equation}

Let \(u_{\rm r},u_{\rm s}:E\times E\to[0,1]\) be Lipschitz functions with \(u_{\rm r}^2+u_{\rm s}^2\equiv 1\). On a filtered probability space
satisfying the usual conditions, let \(B^{1,n},B^{2,n},B^{3,n}\) be
independent \(d\)-dimensional Brownian motions, and abbreviate
\[\begin{aligned}
Q_t^{(n)}&:=\alpha_R(X_t^{(n)}-X_t^{\prime(n)})+\gamma^{-1}(V_t^{(n)}-V_t^{\prime(n)}),\\
u_{{\rm r},t}^{(n)}&:=u_{\rm r}\bigl((X_t^{(n)},V_t^{(n)}),(X_t^{\prime(n)},V_t^{\prime(n)})\bigr),\quad
u_{{\rm s},t}^{(n)}:=u_{\rm s}\bigl((X_t^{(n)},V_t^{(n)}),(X_t^{\prime(n)},V_t^{\prime(n)})\bigr).
\end{aligned}\]
Consider the regularized system
\begin{equation}\label{regu-coup}
\begin{aligned}
\mathrm dX_t^{(n)}
&=
V_t^{(n)}\,\mathrm dt,\\
\mathrm dV_t^{(n)}
&=
-\gamma V_t^{(n)}\,\mathrm dt
-\nabla U(X_t^{(n)})\,\mathrm dt
+\sqrt{2\gamma}\,u_{{\rm r},t}^{(n)}\,\mathrm dB_t^{1,n}
+\sqrt{2\gamma}\,u_{{\rm s},t}^{(n)}\,\mathrm dB_t^{2,n},\\
\mathrm dX_t^{\prime(n)}
&=
V_t^{\prime(n)}\,\mathrm dt,\\
\mathrm dV_t^{\prime(n)}
&=
-\gamma V_t^{\prime(n)}\,\mathrm dt
-\nabla U(X_t^{\prime(n)})\,\mathrm dt
+\sqrt{2\gamma}\,u_{{\rm r},t}^{(n)}\mathsf R_{n,\beta}(Q_t^{(n)})\,\mathrm dB_t^{1,n}\\
&\quad
+\sqrt{2\gamma}\,u_{{\rm s},t}^{(n)}\,\mathrm dB_t^{2,n}
+\sqrt{2\gamma}\,u_{{\rm r},t}^{(n)}\Sigma_{n,\beta}(Q_t^{(n)})\,\mathrm dB_t^{3,n},
\end{aligned}
\end{equation}
with initial state \((x,v,x',v')\).
We identify
\(\Xi_2:=C([0,\infty);E\times E)\) with \(\Xi_1\times\Xi_1\)
through the natural homeomorphism. The following theorem justifies
denoting the solution law of \eqref{regu-coup} by \(\mathbf Q_n\).

\begin{theorem}\label{thwcc}
Suppose that Assumptions~\ref{as4.0}--\ref{as4.3} hold, and fix
\(\beta\in(0,2)\). For every \(n\ge2\), system~\eqref{regu-coup} has
a unique non-explosive strong solution with path marginals
\(\mathbf P^{(x,v)}\) and \(\mathbf P^{(x',v')}\).
The family \(\{\mathbf Q_n\}_{n\ge2}\) is tight on \(\Xi_2\), and
every weak limit belongs to
\(\Pi(\mathbf P^{(x,v)},\mathbf P^{(x',v')})\).
\end{theorem}
\begin{proof}
Fix \(n\ge2\). Assumption \ref{as4.1} implies that \(\nabla U\) is locally
Lipschitz. The preceding construction shows that \(\mathsf R_{n,\beta}\) and
\(\Sigma_{n,\beta}\) are also locally Lipschitz. Hence all coefficients of \eqref{regu-coup} are locally Lipschitz.
To construct the maximal solution, multiply the coefficients by a smooth
cutoff supported in a ball of twice a given radius and equal to one on that
ball. The resulting coefficients are globally Lipschitz, so Picard iteration
gives a unique strong solution \cite[Chapter~5, Theorem~2.9, p.~289]{KS1991}. Pathwise uniqueness \cite[Chapter~5, Theorem~2.5, p.~287]{KS1991} makes solutions with
different cutoff radii agree up to exit from the smaller ball. Pasting these
solutions gives the unique maximal strong solution, with explosion time
\(\tau^{(n)}\).
For \(t<\tau^{(n)}\), define
\[\widetilde B_t^{(n)} :=\int_0^t u_{{\rm r},s}^{(n)}\,\d B_s^{1,n}+ \int_0^t u_{{\rm s},s}^{(n)}\,\d B_s^{2,n},\]
and
\[\widetilde B_t^{\prime(n,\beta)}
:=\int_0^tu_{{\rm r},s}^{(n)}\mathsf R_{n,\beta}(Q_s^{(n)})\,\d B_s^{1,n}+\int_0^tu_{{\rm s},s}^{(n)}\,\d B_s^{2,n} +\int_0^tu_{{\rm r},s}^{(n)}
\Sigma_{n,\beta}(Q_s^{(n)})\,\d B_s^{3,n}.\]
The Brownian motions
\(B^{1,n},B^{2,n},B^{3,n}\) are independent. Moreover,
\((u_{{\rm r},s}^{(n)})^2+(u_{{\rm s},s}^{(n)})^2=1\). Therefore, \(\langle \widetilde B^{(n)}\rangle_t = t\mathbf I_d\).
Furthermore, using \eqref{AAI}, we obtain \(\langle \widetilde B^{\prime(n,\beta)}\rangle_t = t\mathbf I_d\).

To exclude explosion, let \(\tau_j^{(n)}\) be the first exit of the
joint process from the ball of radius \(j\), so that
\(\tau_j^{(n)}\uparrow\tau^{(n)}\), and put
\(\mathscr H(x,v):=U(x)+|v|^2/2\). Although \(U\) is only \(C^1\),
the position paths have finite variation. The ordinary chain rule for
\(U(X)\), together with It\^o's formula for \(|V|^2/2\), therefore gives
\(\mathcal L\mathscr H=\gamma(d-|v|^2)\le\gamma d\).
The stopped stochastic integrands are bounded, hence define true
square-integrable martingales. Consequently,
\[\begin{aligned}
&\mathbb E\Bigl[\mathscr H\left(X_{t\wedge\tau_j^{(n)}}^{(n)},V_{t\wedge\tau_j^{(n)}}^{(n)}\right)
+\mathscr H\left(X_{t\wedge\tau_j^{(n)}}^{\prime(n)},V_{t\wedge\tau_j^{(n)}}^{\prime(n)}\right)\Bigr]\leq \mathscr H(x,v)+\mathscr H(x',v')+2\gamma dt.
\end{aligned}\]
By Assumptions~\ref{as4.0} and~\ref{as4.2}, \(\mathscr H\) is
nonnegative with compact sublevel sets. The preceding bound implies
\(\mathbb P(\tau_j^{(n)}\le T)\to0\) for every \(T<\infty\),
and hence \(\tau^{(n)}=\infty\) almost surely.
L\'evy's characterization
\cite[Chapter~3, Theorem~3.16]{KS1991} now identifies both
\(\widetilde B^{(n)}\) and \(\widetilde B^{\prime(n,\beta)}\) as
Brownian motions in the common filtration. Thus both coordinates solve
\eqref{eq4.1} globally. 

The same local construction and energy bound apply to the one-particle
equation. For each cutoff, its Picard solution is determined measurably
by the initial state and the Brownian path; pathwise uniqueness forces
any weak solution to agree with it up to the cutoff exit time.
Non-explosion then gives uniqueness of the full path law. By uniqueness in law for \eqref{eq4.1}, the two marginal path laws of \(\mathbf Q_n\) are independent of \(n\). More precisely, for every \(n\ge2\) and \(A\in\mathcal B(\Xi_1)\),
\begin{equation}\label{Qn}
	\mathbf Q_n(A\times\Xi_1)=\mathbf P^{(x,v)}(A),
	\qquad\mathbf Q_n(\Xi_1\times A)=\mathbf P^{(x',v')}(A).
\end{equation}
The space \(\Xi_1\), equipped with the topology of uniform convergence on compact
time intervals, is Polish.
Given \(\epsilon>0\), choose compact sets \(K_{1, \epsilon},K_{2, \epsilon}\subset\Xi_1\)
with marginal masses at least \(1-\epsilon/2\). Then
\(\mathbf Q_n((K_{1, \epsilon}\times K_{2, \epsilon})^c)\le\epsilon\) for every \(n\),
which proves tightness of \(\{\mathbf Q_n\}_{n\ge2}\). Prokhorov's theorem
\cite[Chapter~2, Theorem~4.7, p.~62]{KS1991} gives weakly convergent
subsequences. Let \(\mathbf Q_{n_k}\) converge weakly to \(\mathbf Q\) on \(\Xi_2\).

It remains to verify the marginal laws of \(\mathbf Q\). The two coordinate maps
\[\Xi_2\ni(\omega,\omega')\mapsto\omega\in\Xi_1,
\qquad\Xi_2\ni(\omega,\omega')\mapsto\omega'\in\Xi_1\]
are continuous. Hence, the weak convergence of \(\mathbf Q_{n_k}\) to \(\mathbf Q\) implies
that the marginals of \(\mathbf Q_{n_k}\) converge weakly to the corresponding marginals of \(\mathbf Q\). By \eqref{Qn}, for every \(k\), these marginals are
\(\mathbf P^{(x,v)}\) and \(\mathbf P^{(x',v')}\), respectively. Hence these are
also the marginals of \(\mathbf Q\).
Thus, \(\mathbf Q\in
\Pi\bigl(\mathbf P^{(x,v)},\mathbf P^{(x',v')}\bigr)\). This completes the proof.
\end{proof}
\par The next lemma identifies a semimartingale structure retained by every
weak limit.
\begin{lemma}\label{semimartingale}
Suppose that Assumptions~\ref{as4.0}--\ref{as4.3} hold, fix
\(\beta\in(0,2)\), and let \(\mathbf Q\) be any weak limit of the
laws \(\mathbf Q_n\) from \eqref{regu-coup}. Write
\[
W_t^{(n)}=\bigl((X_t^{(n)},V_t^{(n)}),
(X_t^{\prime(n)},V_t^{\prime(n)})\bigr),\qquad
W_t=\bigl((X_t,V_t),(X'_t,V'_t)\bigr)
\]
for the canonical process under \(\mathbf Q_n\) and \(\mathbf Q\),
respectively. On the canonical path space \(\Xi_2\), define the raw filtration
generated by the joint coordinate process \(W\) by \(\mathcal G_t^0:=\sigma(W_r:0\le r\le t)\).
Let \(\mathcal N_{\mathbf Q}
:=
\left\{
A\subseteq\Xi_2:
A\subseteq B,~ B\in\mathcal B(\Xi_2)
\text{ with }\mathbf Q(B)=0
\right\}\).
We extend \(\mathbf Q\) to
\(\sigma(\mathcal B(\Xi_2)\cup\mathcal N_{\mathbf Q})\)
by completion and retain the same notation for this extension.
The usual augmentation of \((\mathcal G_t^0)\) under
\(\mathbf Q\) is \(\mathcal G_t^{\mathbf Q}
:=
\bigcap_{r>t}
\sigma\bigl(\mathcal G_r^0\cup\mathcal N_{\mathbf Q}\bigr)\), \(t\ge0\).
Set
\begin{equation}\label{qb}
\begin{aligned}
q\bigl((x,v),(x',v')\bigr)\!:&=\alpha_R(x-x')+\gamma^{-1}(v-v'),\\
b\bigl((x,v),(x',v')\bigr)\!:&=(\alpha_R-1)(v-v')-\gamma^{-1}\bigl(\nabla U(x)-\nabla U(x')\bigr)\\
&=\!\gamma(\alpha_R-1)q\bigl((x,v),(x',v')\bigr)-\gamma\alpha_R(\alpha_R-1)(x-x')
-\gamma^{-1}\bigl(\nabla U(x)-\nabla U(x')\bigr).
\end{aligned}
\end{equation}
Then \(Q_t:=q(W_t)\) has the decomposition
\(M_t:=Q_t-Q_0-\int_0^t b(W_s)\,\mathrm ds,\)
where \(M\) is a continuous square-integrable \(\mathcal G_t^{\mathbf Q}\)-martingale
with \(\mathbb E_{\mathbf Q}|M_t|^2\le8d\gamma^{-1}t\).
Its bracket admits a \((\mathcal G_t^{\mathbf Q})\)-predictable density
\(m_t\in\mathbb S_d^+\) satisfying
\(\langle M\rangle_t=\int_0^t m_s\,\mathrm ds\) and
\(0\preceq m_t\preceq8\gamma^{-1}\mathbf I_d\)
for \(\mathrm dt\otimes\mathbf Q\)-almost every \((t,\omega)\).
\end{lemma}
\begin{proof}
Pass to a subsequence that converges weakly to \(\mathbf Q\), and relabel it as \(\{\mathbf Q_n\}_{n\ge2}\). Subtracting the two velocity equations in \eqref{regu-coup}, combining the result with \(Q^{(n)}=\alpha_R(X^{(n)}-X^{\prime(n)})+\gamma^{-1}(V^{(n)}-V^{\prime(n)})\), and using the definition of \(b\) in \eqref{qb} gives
\[\begin{aligned}
\mathrm dQ_t^{(n)}={}&b(W_t^{(n)})\,\mathrm dt+\sqrt{2\gamma^{-1}}\,u_{{\rm r},t}^{(n)}\bigl(\mathbf I_d-\mathsf R_{n,\beta}(Q_t^{(n)})\bigr)\,\mathrm dB_t^{1,n}-\sqrt{2\gamma^{-1}}\,u_{{\rm r},t}^{(n)}\Sigma_{n,\beta}(Q_t^{(n)})\,\mathrm dB_t^{3,n}.
\end{aligned}\]
The sum of the two stochastic integrals defines a continuous martingale \(M^{(n)}\). Symmetry of \(\mathsf R_{n,\beta}\) and \eqref{AAI} imply 
\[\bigl(\mathbf I_d-\mathsf R_{n,\beta}(q)\bigr)^2+\Sigma_{n,\beta}(q)\Sigma_{n,\beta}(q)^{\mathrm T}=2\bigl(\mathbf I_d-\mathsf R_{n,\beta}(q)\bigr)=4\theta_{n,\beta}\mathbf e_n(q)\mathbf e_n(q)^{\mathrm T}.\]
The independence of \(B^{1,n}\) and \(B^{3,n}\), together with \(0<\theta_{n,\beta}<1\), \(0\leq u_{{\rm r},t}^{(n)}\leq1\), and \(|\mathbf e_n(q)|\leq1\), therefore yields
\begin{equation}\label{bMn}
\mathrm d\langle M^{(n)}\rangle_t=8\gamma^{-1}\theta_{n,\beta}(u_{{\rm r},t}^{(n)})^2\mathbf e_n(Q_t^{(n)})\mathbf e_n(Q_t^{(n)})^{\mathrm T}\,\mathrm dt\preceq8\gamma^{-1}\mathbf I_d\,\mathrm dt.
\end{equation}
Thus \(M^{(n)}\) is square-integrable on every finite interval.

Fix \(T>0\). The map
\[
\Gamma_T(w)(t):=q(w(t))-q(w(0))-\int_0^t b(w(s))\,\mathrm ds
\]
from \(C([0,T];E^2)\) to \(C([0,T];\mathbb R^d)\) is continuous:
uniformly convergent paths have ranges in a common compact set, on
which the continuous function \(b\) is uniformly continuous.
Applying the continuous mapping theorem to \((\operatorname{id},\Gamma_T):
w\longmapsto \bigl(w,\Gamma_T(w)\bigr)\),
and noting that \(M^{(n)}=\Gamma_T(W^{(n)})\) and \(M=\Gamma_T(W)\), we obtain the following convergence for every \(\Psi\in C_b\bigl(C([0,T];\mathbb R^{4d})\times C([0,T];\mathbb R^d)\bigr)\),
\begin{equation}\label{Qnw}
\lim_{n\to\infty}\mathbb E_{\mathbf Q_n}\left[\Psi\bigl(W^{(n)},M^{(n)}\bigr)\right]=\mathbb E_{\mathbf Q}\left[\Psi(W,M)\right].
\end{equation}

Let \(0\le t_1<\cdots<t_k\le s<t\le T\) and let \(\varphi\in
C_b\bigl((\mathbb R^{4d})^k\bigr)\).
For each component \(i\), the martingale property of
$M^{(n)}$ under $\mathbf Q_n$ gives
\begin{equation}\label{Mn-martingale}
\mathbb E_{\mathbf Q_n}\left[\bigl(M_t^{(n),i}-M_s^{(n),i}\bigr)\varphi(W_{t_1}^{(n)},\cdots,W_{t_k}^{(n)})\right]=0.
\end{equation}
By \eqref{bMn}, \(\mathbb E_{\mathbf Q_n}|M_t^{(n)}-M_s^{(n)}|^2
\le 8d\gamma^{-1}(t-s)\).
Therefore, the family \(\bigl\{
\bigl(M_t^{(n),i}-M_s^{(n),i}\bigr)
\varphi(W_{t_1}^{(n)},\cdots,W_{t_k}^{(n)})\bigr\}_{n\ge 2}\) is uniformly integrable. Combining
\eqref{Qnw}, the uniform integrability above, and
\eqref{Mn-martingale}, we obtain
\begin{equation}\label{martingale-limit}
\mathbb E_{\mathbf Q}
\left[\bigl(M_t^i-M_s^i\bigr)\varphi(W_{t_1},\cdots,W_{t_k})\right]=0.
\end{equation}
The bounded continuous past cylinder functions form an algebra generating \(\mathcal G_s^0\). Since \(M_t^i-M_s^i\) is integrable, dominated convergence makes the class on which \eqref{martingale-limit} holds closed under bounded monotone limits. The functional monotone-class theorem therefore extends the identity to every bounded $\mathcal G_s^0$-measurable random variable
$\xi$, \(\mathbb E_{\mathbf Q}
\left[(M_t^i-M_s^i)\xi\right]=0\).
Hence each component $M^i$ is an
$(\mathcal G_t^0)$-martingale under $\mathbf Q$, and therefore $M$ is a continuous martingale for the raw canonical filtration. Moreover, apply \eqref{Qnw} first to the bounded continuous function \(|M_t|^2\wedge N\) and then let \(N\uparrow\infty\) by monotone convergence. This gives
\begin{equation*}
\begin{aligned}
\mathbb{E}_{\mathbf Q}[|M_t|^2] &\leq \liminf_{n \to \infty} \int_{C([0,T];\mathbb{R}^d)} |w(t)|^2 \operatorname{Law}_{\mathbf Q_n}((M_s^{(n)})_{0\le s\le T})(\d w) \\
&=\liminf_{n\to\infty}\mathbb E_{\mathbf Q_n}|M_t^{(n)}|^2=\liminf_{n\to\infty}
\mathbb E_{\mathbf Q_n}\operatorname{tr}\langle M^{(n)}\rangle_t\le8d\gamma^{-1}t.
\end{aligned}
\end{equation*}
Thus \(M\) is square-integrable on every finite time interval.
Completion by \(\mathbf Q\)-null sets preserves its martingale property.
For the right-continuous augmentation, fix \(s<t\), choose
\(s_k\downarrow s\) with \(s<s_k<t\), and take \(A\in\mathcal G_s^{\mathbf Q}\).
Then \(A\) belongs to the completed raw sigma-field at every \(s_k\), so \(\mathbb E_{\mathbf Q}\bigl[(M_t-M_{s_k})\mathbf1_A\bigr]=0\).
Doob's maximal inequality
\cite[Chapter~1, Theorem~3.8(iv), p.~14]{KS1991} gives \(\mathbb E_{\mathbf Q}\sup_{r\le T}|M_r|^2<\infty\) for \(T>t\).
Continuity and this integrable bound imply \(M_{s_k}\to M_s\) in \(L^1\).
Passing to the limit in the preceding identity proves that \(M\) is a
continuous square-integrable \((\mathcal G_t^{\mathbf Q})\)-martingale.

It remains to control the bracket.
Fix \(u\in\mathbb Q^d\), and define \(N_t^{(n),u}:=u^{\mathrm T} M_t^{(n)}\) and \(N_t^u:=u^{\mathrm T} M_t\).
Under $\mathbf Q_n$, \(\langle N^{(n),u}\rangle_t= u^{\mathrm T} \langle M^{(n)}\rangle_t u\).
It follows from \eqref{bMn} that
\(\mathrm d\langle N^{(n),u}\rangle_t\le 8\gamma^{-1}|u|^2\,\mathrm dt.\)
Consequently, the process \(S_t^{(n),u}:=\bigl(N_t^{(n),u}\bigr)^2
-8\gamma^{-1}|u|^2t\)
is a $\mathbf Q_n$-supermartingale.
The Burkholder--Davis--Gundy inequality \cite[Chapter~3, Theorem~3.28, p.~166]{KS1991} gives
\begin{equation}\label{eq:uniform-L4}
\sup_{n\ge2}\mathbb E_{\mathbf Q_n}
\left|N_t^{(n),u}\right|^4\le C\sup_{n\ge2}
\mathbb E_{\mathbf Q_n}\left(\langle N^{(n),u}\rangle_t\right)^2\le C\gamma^{-2}|u|^4t^2.
\end{equation}
Thus, for every fixed $t<\infty$, the family \(\bigl\{\bigl(N_t^{(n),u}\bigr)^2\bigr\}_{n\ge 2}\) is uniformly integrable.

Fix \(T>0\), and set
\(S_t^u:=(N_t^u)^2-8\gamma^{-1}|u|^2t\).
For \(0\le t_1<\cdots<t_k\le s<t\le T\), let
\(\varphi\in C_b((\mathbb R^{4d})^k)\) be nonnegative.
Since $S^{(n),u}$ is a $\mathbf Q_n$-supermartingale, by \eqref{Qnw},
\eqref{eq:uniform-L4}, and uniform integrability, we obtain
\(\mathbb E_{\mathbf Q} \left[\bigl(S_t^u-S_s^u\bigr) \varphi(W_{t_1},\ldots,W_{t_k})\right]\le0.\)
A monotone-class argument extends this inequality to every bounded,
nonnegative \(\mathcal G_s^0\)-measurable random variable.
Thus \(S^u\) is an \((\mathcal G_t^0)\)-supermartingale under
\(\mathbf Q\).

We next verify that this property is preserved under \((\mathcal G_t^{\mathbf Q})\).
Since \(N^u\) is a continuous square-integrable martingale,
Doob's inequality
\cite[Chapter~1, Theorem~3.8(iv), p.~14]{KS1991} gives
\[
\sup_{r\le T}|S_r^u|
\le \sup_{r\le T}|N_r^u|^2+8\gamma^{-1}|u|^2T,
\qquad
\mathbb E_{\mathbf Q}\sup_{r\le T}|N_r^u|^2<\infty.
\]
Completion by \(\mathbf Q\)-null sets preserves the
supermartingale property. For \(A\in\mathcal G_s^{\mathbf Q}\), choose
\(s_j\downarrow s\) with \(s<s_j<t\).
By the definition of the usual augmentation, \(A\) belongs to
the completed raw sigma-field at every \(s_j\), and hence \(\mathbb E_{\mathbf Q}
\left[(S_t^u-S_{s_j}^u)\mathbf1_A\right]\le0\).
Continuity and the integrable bound above allow us to let
\(j\to\infty\) by dominated convergence. Therefore \(S^u\) is
an \((\mathcal G_t^{\mathbf Q})\)-supermartingale under \(\mathbf Q\).
The same bound also ensures uniform integrability of its values
at all \((\mathcal G_t^{\mathbf Q})\)-stopping times bounded by \(T\), as
needed for the Doob--Meyer decomposition below.

To justify the sign of its finite-variation part, apply the
Doob--Meyer theorem to \(-S^u\) on \([0,T]\)
\cite[Chapter~1, Theorem~4.10, pp.~24--25]{KS1991}.
Its compensator is continuous: the same integrable bound and path
continuity imply convergence of expected values along convergent
bounded stopping times, hence regularity; the continuity criterion
in \cite[Chapter~1, Theorem~4.14, p.~28]{KS1991} applies.
Thus \(S_t^u=L_t^u-A_t^u\), \(0\le t\le T\),
where \(L^u\) is a continuous martingale, \(A^u\) is continuous and
nondecreasing, and both start at zero.
Comparing this decomposition with
\((N^u)^2-\langle N^u\rangle\), and using uniqueness of the
continuous finite-variation compensator
\cite[Chapter~1, Theorem~5.13 and its proof, pp.~35--36]{KS1991}, gives
\(\langle N^u\rangle_t=8\gamma^{-1}|u|^2t-A_t^u\).
In particular, \(8\gamma^{-1}|u|^2t-u^{\mathrm T}\langle M\rangle_tu\)
is nondecreasing. Since \(T\) was arbitrary, this conclusion holds on
every bounded interval.
Hence, for every \(0\le s\le t\),
\begin{equation}\label{quadratic-M}
u^{\mathrm T}(\langle M\rangle_t-\langle M\rangle_s)u
\le8\gamma^{-1}|u|^2(t-s),\qquad u\in\mathbb Q^d.
\end{equation}
Since \(\mathbb Q^d\) is countable, we may choose one event of full
\(\mathbf Q\)-probability on which the preceding monotonicity holds for every
\(u\in\mathbb Q^d\). By continuity in \(u\) and density of \(\mathbb Q^d\) in
\(\mathbb R^d\), \eqref{quadratic-M} then holds for every
\(u\in\mathbb R^d\), simultaneously for all \(0\le s\le t\). Since
\(\langle M\rangle_t-\langle M\rangle_s\) is symmetric and non-negative
definite, we conclude that
\begin{equation*}
0\preceq\langle M\rangle_t-\langle M\rangle_s
\preceq8\gamma^{-1}(t-s)\mathbf I_d,\qquad0\le s\le t.
\end{equation*}
To justify the assertion entry by entry, apply the scalar inequality to
$u=e_i$ and $u=e_i+e_j$, where $e_i$ are the coordinate vectors. The
diagonal brackets are locally absolutely continuous, and polarization gives \(\langle M^i,M^j\rangle
=\frac{1}{2}\bigl(\langle M^i+M^j\rangle-
\langle M^i\rangle-\langle M^j\rangle\bigr)\).
Thus every entry of \(\langle M\rangle\) is locally absolutely continuous.
The density can be chosen predictably by an explicit construction.
Write \(A_t=\langle M\rangle_t\) and
\(m_t^{(k)}=k(A_t-A_{(t-1/k)^+})\).
Each \(m^{(k)}\) is continuous and adapted, hence predictable, and satisfies
\(0\preceq m^{(k)}\preceq8\gamma^{-1}\mathbf I_d\).
Define \(m_t\) as the entrywise limit where all entries converge, and as the
zero matrix otherwise. This is predictable and has the same matrix bound.
For each path on which \(A\) is locally absolutely continuous, the Lebesgue
differentiation theorem identifies this limit with its derivative at almost
every time. Fubini's theorem gives the identification
\(\mathrm dt\otimes\mathbf Q\)-almost everywhere. Consequently, \(\langle M\rangle_t=\int_0^t m_s\,\mathrm ds\).
The matrix inequality above implies \(0\preceq m_t\preceq8\gamma^{-1}\mathbf I_d\) for \(\mathrm dt\otimes\mathbf Q\text{-a.e. }(t,\omega)\).
\end{proof}

We finally choose the switches used in Section~\ref{Sec4}. Fix \(\delta\in(0,1)\), and set
\[
S_{R,\delta}:=\{(x,v)\in E:|x|<R^{2/m}-\delta,\ |v|<R-\delta\}.
\]
Writing \(z=x-x'\) and \(y=v-v'\), we require
\begin{equation}\label{rc}
\begin{aligned}
u_{\rm r}((x,v),(x',v'))&=0
\quad\text{if}\quad
(x,v)\notin S_R
\ \text{or}\
(x',v')\notin S_R
\ \text{or}\
|z|^2+|y|^2\leq \frac{\delta^2}{4},\\
u_{\rm r}((x,v),(x',v'))&=1
\quad\text{if}\quad
(x,v)\in S_{R,\delta},\
(x',v')\in S_{R,\delta},
\ \text{and}\
|z|^2+|y|^2\geq \delta^2,
\end{aligned}
\end{equation}
where \(S_R\) is defined in \eqref{SRset}. For an explicit construction, define, for \(s\ge0\),
\[
\vartheta_1(s):=1\wedge\frac{s}{\delta},\qquad
\vartheta_2(s):=1\wedge\left(\frac{2s}{\delta}-1\right)^+,
\quad
d_R(x,v):=(R^{2/m}-|x|)^+\wedge(R-|v|)^+,
\]
and put
\[
\vartheta_{R,\delta}(\xi,\xi')
:=\vartheta_1(d_R(\xi))\vartheta_1(d_R(\xi'))
\vartheta_2(|\xi-\xi'|),\qquad \xi,\xi'\in E.
\]
Then
\[
u_{\rm r}(\xi,\xi')
:=\sin\bigl(\tfrac\pi2\vartheta_{R,\delta}(\xi,\xi')\bigr),\qquad
u_{\rm s}(\xi,\xi')
:=\cos\bigl(\tfrac\pi2\vartheta_{R,\delta}(\xi,\xi')\bigr)
\]
are Lipschitz, take values in \([0,1]\), and satisfy both
\(u_{\rm r}^2+u_{\rm s}^2=1\) and \eqref{rc}.
The gap between \(S_{R,\delta}\) and \(S_R^c\), together with the two separation thresholds, allows the required continuous transitions.

\section{Auxiliary lemmas}\label{Sec4}
This section establishes the localization bounds and stochastic differential inequalities used in the contraction proof.

\subsection{Lyapunov and localization estimates}
All constants and functions are those of Section~\ref{Sec2}.
\begin{proposition}\label{pr4.1}
Suppose that Assumptions \ref{as4.0}--\ref{as4.3} hold. Then
\begin{equation*}
(\mathcal LF)(x,v)\le C_5-F(x,v)G(x,v),\qquad (x,v)\in E,
\end{equation*}
and \(\mathcal V(x,v)\ge U(x)/4+c_0(|x|^2+|v|^2)\).
\end{proposition}
\begin{proof}
Although $F$ need not be twice differentiable in the position variable,
it is $C^1$ in $x$ and $C^2$ in $v$. This is sufficient for It\^o's
formula because the position process has finite variation.
Applying \eqref{eq4.3} to \eqref{lyapunov-functions} gives
\[(\mathcal LF)(x,v)=-\frac\gamma8F(x,v)\mathcal D_\theta(x,v),\]
where \(\mathcal D_\theta(x,v)
:=\theta\langle\nabla U(x),x\rangle+D_\theta(x,v)\) and
\[D_\theta(x,v)
:=
(1-\theta)\langle\nabla U(x),x\rangle
+|v|^2
-\frac{1}{8}|2v+\gamma x|^2
-2d.\]
We first consider the region \(|x|\ge K_1\). By Assumption~\ref{as4.2}, \eqref{a}, and Young's inequality
\[\gamma\langle x,v\rangle\le (1-\theta)L_2|x|^2+\frac{\gamma^2}{4(1-\theta)L_2}|v|^2,\]
we have
\[D_\theta(x,v)\ge\frac{\lambda_\theta}{2}\bigl(|x|^2+|v|^2\bigr)-2d.\]
We next consider the region \(|x|<K_1\). Using
\[-\frac{\gamma}{2}\langle x,v\rangle \ge -\frac{1}{4}|v|^2-\frac{\gamma^2}{4}|x|^2,\]
we obtain
\begin{align*}
D_\theta(x,v)
&=
(1-\theta)\langle\nabla U(x),x\rangle
+\frac{1}{2}|v|^2
-\frac{\gamma}{2}\langle x,v\rangle
-\frac{\gamma^2}{8}|x|^2
-2d\\
&\ge
\frac{1}{4}|v|^2
+
(1-\theta)\langle\nabla U(x),x\rangle
-\frac{3\gamma^2}{8}|x|^2
-2d.
\end{align*}
The bound in Remark \ref{re4.1} implies, for \(|x|\leq K_1\),
\[(1-\theta)\langle\nabla U(x),x\rangle
\geq-(1-\theta)K_1C_1(1+K_1^{\ell+1}).\]
Consequently,
\(D_\theta(x,v) \ge \frac{1}{4}\bigl(|x|^2+|v|^2\bigr)-C_2,\)
where $C_2$ is defined in \eqref{global-C-actual}.
Since $C_2\ge2d$, the two regions give, for all $(x,v)\in E$,
\[D_\theta(x,v)\ge \min\left\{\frac{\lambda_\theta}{2}, \frac{1}{4}\right\}\bigl(|x|^2+|v|^2\bigr)-C_2.\]
Combining this with Assumption \ref{as4.3}, we get
\[\mathcal D_\theta(x,v)
\ge\theta L_3 U(x)
+\min\left\{\frac{\lambda_\theta}{2}, \frac{1}{4}\right\}\bigl(|x|^2+|v|^2\bigr)
-C_2-\theta K_2.\]
Moreover, observe that
\[\frac{\gamma^2}{2}|x|^2+\gamma\langle x,v\rangle+|v|^2
\le
\max\left\{\gamma^2, \frac{3}{2}\right\}\bigl(|x|^2+|v|^2\bigr).\]
Hence
\[\mathcal V(x,v)
\le
\frac{1}{4} U(x)
+\frac{1}{16}
\max\left\{2\gamma^2, 3\right\}\bigl(|x|^2+|v|^2\bigr),\]
which implies that
\[1+\mathcal V(x,v)\le\frac{1}{8}
\max\{\gamma^2, 8\}\left(1+U(x)+|x|^2+|v|^2\right).\]
By \eqref{a} and \eqref{global-C-actual},
\[\mathcal D_\theta(x,v)\ge \frac{16a}{\gamma}\left(1+\mathcal V(x,v)\right)-C_3,
\qquad (x,v)\in\mathbb{R}^{2d}.\]
Here \(C_3\) is defined in \eqref{global-C-actual}.
Then
\[(\mathcal L F)(x,v)+F(x,v)G(x,v)
\le
F(x,v)
\left[
-a(1+\mathcal V(x,v))
+
\frac{\gamma C_3}{8}
\right].\]
Recall from \eqref{global-C-actual} that \(C_4=\gamma C_3/8\).
By Young's inequality, we have
\begin{equation}\label{V1}
\mathcal V(x,v)\ge \frac{U(x)}4+c_0(|x|^2+|v|^2).
\end{equation}
In particular, \(\mathcal V\geq0\). We now maximize the preceding scalar upper
bound. Consider the scalar function \(u\mapsto\mathrm e^u[C_4-a(1+u)]\) on \([0,\infty)\); its derivative is \(\mathrm e^u[C_4-2a-au]\). If \(C_4\leq a\), the function is nonpositive. If \(a<C_4<2a\), it is decreasing and its supremum is \(C_4-a\). If \(C_4\geq2a\), its unique maximizer is \(u=C_4/a-2\), where its value is \(a\exp(C_4/a-2)\). Consequently,
\[\sup_{u\geq0}\left[\mathrm e^u\{C_4-a(1+u)\}\right]^+=\Psi(a,C_4)\leq C_5.\]
This proves the drift estimate. The choice $C_5\ge a>0$ also makes the logarithms in \eqref{radius-1} and \eqref{radius-45} positive and the normalization \eqref{epsilon} well defined.
\end{proof}

Theorem~\ref{thwcc} gives the non-explosive strong solution for every deterministic initial state. For an initial state independent of the Brownian motion, the same conclusion follows by pasting solutions on bounded initial-data events and conditioning the stopped energy estimate on that state; no moment assumption is needed.

Since $G\ge a$, Proposition~\ref{pr4.1} gives $\mathcal LF+aF\le C_5$. The stopped It\^o formula with integrating factor $e^{at}$, followed by Fatou's lemma and nonexplosion, therefore yields
\[
P_tF(\xi)\le e^{-at}F(\xi)+\frac{C_5}{a}(1-e^{-at}),\qquad \xi\in E,\quad t\ge0.
\]
Integration against the initial law shows that a finite initial $F$-moment remains finite at every finite time.

\begin{lemma}\label{le4.1}
Under Assumptions \ref{as4.0}--\ref{as4.3}, the set \(\mathcal H\) is compact
and
\[\mathcal H\subset \{(x,v):|x|^2+|v|^2\le R_1^2\}^2.\]
\end{lemma}
\begin{proof}
Let \(\bigl((x,v),(x',v')\bigr)\in \mathcal H\). Both summands in the definition of
\(\mathcal H\) are nonnegative. Hence \(F(x,v)G(x,v)\leq2C_5\).
Since \(G=a(1+\mathcal V)\geq a\),
\(a\exp\{\mathcal V(x,v)\}\leq F(x,v)G(x,v)\leq2C_5.\)
Taking logarithms and using \eqref{V1} gives
\[c_0(|x|^2+|v|^2)\leq\mathcal V(x,v)
\leq\log\frac{2C_5}{a}=c_0R_1^2.\]
Thus \(|x|^2+|v|^2\leq R_1^2\). The same argument applies to
\((x',v')\). The set \(\mathcal H\) is closed because \(F\) and \(G\) are
continuous. It is bounded by the preceding inclusion. Hence it is compact.
\end{proof}

The lower bound \eqref{eq:power-lower} in Assumption~\ref{as4.0} and \eqref{V1} give the energy bound needed for the anisotropic localization:
\begin{equation}\label{energy-tail-bound}
 \mathcal V(x,v)\ge \bigl(b_0/4\wedge c_0\bigr)\bigl(|x|^m+|v|^2\bigr)-B_0/4,
 \qquad (x,v)\in E.
\end{equation}
Indeed, the coefficient of $|x|^m$ is at least $b_0/4$, the coefficient of $|v|^2$ is at least $c_0$, and the remaining quadratic position term in \eqref{V1} is nonnegative.
The next estimate uses \eqref{energy-tail-bound} to control the local force modulus whenever one state lies outside the box in \eqref{SRset}.
\begin{lemma}\label{le4.2}
Suppose that Assumptions~\ref{as4.0}--\ref{as4.3} hold.
For every $R\ge R_2$,
\[
\gamma\bigl(A_R+\alpha_R-1\bigr)+1+
\frac{|\nabla U(x)-\nabla U(x')|}{\gamma|x-x'|}
\le\frac16\bigl(G(x,v)\vee G(x',v')\bigr)
\]
whenever at least one of $(x,v),(x',v')$ lies outside $S_R$.
The quotient is defined to be zero when $x=x'$.
\end{lemma}
\begin{proof}
Set $s:=\max\{|x|^m+|v|^2,|x'|^m+|v'|^2\}$.
By \eqref{SRset}, $s\ge R^2\ge1$, so
$|x|^\ell,|x'|^\ell\le s^{\ell/m}$ and $R^{\ell/m}\le s^{\ell/m}$.
Assumption~\ref{as4.1} and \eqref{eq4.5} therefore imply
\begin{align*}
\gamma\bigl(A_R+\alpha_R-1\bigr)+1+
\frac{|\nabla U(x)-\nabla U(x')|}{\gamma|x-x'|}
&\le\gamma+1+3\sqrt{3L_1}\,R^{\ell/m}
       +\frac{L_1}{\gamma}(1+2s^{\ell/m})\\
&\le\gamma+1+
 \left(3\sqrt{3L_1}+\frac{3L_1}{\gamma}\right)s^{\ell/m}.
\end{align*}
This bound also holds when $x=x'$ under the stated convention.
The definition \eqref{radius-2} makes each of the two terms on the last line at most $a\bigl(b_0/4\wedge c_0\bigr)s/24$.
Moreover, \eqref{energy-tail-bound} and $R_2^2\ge B_0/[2(b_0/4\wedge c_0)]$ give
\[
G(x,v)\vee G(x',v')
\ge a(1+\bigl(b_0/4\wedge c_0\bigr)s-B_0/4)\ge\frac{a\bigl(b_0/4\wedge c_0\bigr)s}{2}.
\]
The required inequality follows by comparing these bounds.
\end{proof}

The velocity gradient of $F$ is $\nabla_vF(x,v)=F(x,v)(2v+\gamma x)/8$.
Since $m\ge2$ and $R\ge1$, the position radius $R^{2/m}$ is at most $R$.
Thus the first scale in \eqref{eq4.7} satisfies
\begin{equation}\label{velocity-weight-bound}
2\sup_{(x,v)\in S_R}\frac{|\nabla_vF(x,v)|}{F(x,v)}
\le\frac{\gamma R^{2/m}+2R}{4}
\le\frac{\gamma\vee2}{2}R=\eta_{1,R}.
\end{equation}
By \eqref{rxv}, Lemma~\ref{le4.1} and the definition of $\eta_{2,R}$ in \eqref{eq4.7} also give
\begin{equation}\label{eta2}
r((x,v),(x',v'))
\le(A_R+\alpha_R)|x-x'|+\gamma^{-1}|v-v'|
\le\eta_{2,R},\qquad ((x,v),(x',v'))\in\mathcal H.
\end{equation}
The next bound compares the reciprocal weight, which depends on this core scale, with the energy outside $S_R$.

\begin{lemma}\label{le4.3}
Suppose that Assumptions~\ref{as4.0}--\ref{as4.3} hold.
For every $R\ge R_3$,
\[\frac{2C_5}{\eta_{3,R}}\le F(x,v)\vee F(x',v')\]
whenever at least one of the two states lies outside $S_R$.
\end{lemma}
\begin{proof}
For $R\ge1$, \eqref{eq4.5}, \eqref{eq4.7}, and \eqref{radius-majorants} give
\[
\begin{aligned}
A_R&\le(1+2\sqrt{3L_1}/\gamma)R^{\ell/m},\qquad \alpha_R\le(1+\sqrt{3L_1}/\gamma)R^{\ell/m},\\
\bigl(A_R+\alpha_R-1\bigr)&\le(1+3\sqrt{3L_1}/\gamma)R^{\ell/m},
\qquad \eta_{2,R}\le\widehat A_1R^{\ell/m}.
\end{aligned}
\]
Substitution into \eqref{Seeq1} yields
\(h_R(\eta_{2,R})\le\widehat A_2R^{3\ell/m}+\widehat A_3R^{1+\ell/m}.\)
Since $h_R$ is increasing and $\Phi_R(s)\le s$, \eqref{J1-definition} gives
\[
J_{1,R}\le\gamma\int_0^{\eta_{2,R}}s e^{h_R(s)}\,\mathrm ds
\le\frac{\gamma\eta_{2,R}^2}{2}e^{h_R(\eta_{2,R})}.
\]
Using \eqref{eq4.8} and \eqref{radius-majorants}, we obtain
$\eta_{3,R}^{-1}\le C_6R^{2\ell/m}e^{h_R(\eta_{2,R})}$ and hence
\begin{equation}\label{profile-tail-balance}
B_0/4+\log\frac{2C_5}{\eta_{3,R}}
\le B_0/4+C_7+\frac{2\ell}{m}\log R
  +\widehat A_2R^{3\ell/m}+\widehat A_3R^{1+\ell/m}\le \bigl(b_0/4\wedge c_0\bigr)R^2.
\end{equation}
For the last inequality, \eqref{radius-3} bounds each of the four terms on the first line by $\bigl(b_0/4\wedge c_0\bigr)R^2/4$.
Specifically, the constant term uses $R^2\ge(B_0+4C_7)/(b_0/4\wedge c_0)$;
the logarithmic term uses $\log R\le R$ and $R\ge8\ell/[m(b_0/4\wedge c_0)]$;
the last two terms use the powers $m/(2m-3\ell)$ and $m/(m-\ell)$.
Both denominators are positive by \eqref{growth-regime}.
Exponentiating \eqref{profile-tail-balance} proves the first bound.
If one state lies outside $S_R$, \eqref{SRset} and \eqref{energy-tail-bound} give
$\mathcal V(x,v)\vee\mathcal V(x',v')\ge \bigl(b_0/4\wedge c_0\bigr)R^2-B_0/4$, which proves the second bound.
\end{proof}

Lemmas~\ref{le4.2} and~\ref{le4.3} yield the exterior estimate used in Case I below.
\begin{proposition}\label{pr4.2}
Under Assumptions~\ref{as4.0}--\ref{as4.3}, the number $R_\star$ in \eqref{Rstar-explicit} is finite.
For every $R\ge R_\star$, $\mathcal H\subset S_R\times S_R$.
Moreover, whenever at least one state lies outside $S_R$,
\begin{equation}\label{eqpr42}
\begin{aligned}
&2C_5+\left(\frac{2C_5}{\eta_{3,R}}+F(x,v)+F(x',v')\right)
\left(\gamma\bigl(A_R+\alpha_R-1\bigr)+1+
\frac{|\nabla U(x)-\nabla U(x')|}{\gamma|x-x'|}\right)\\
\le{}&F(x,v)G(x,v)+F(x',v')G(x',v').
\end{aligned}
\end{equation}
The quotient is zero when $x=x'$.
\end{proposition}
\begin{proof}
All constants in \eqref{Rstar-explicit} are determined by the model parameters and the power-growth constants in \eqref{eq:power-lower}.
The powers $m/[2(m-\ell)]$, $m/(2m-3\ell)$ and $m/(m-\ell)$ in \eqref{radius-2}--\eqref{radius-3} have positive denominators by \eqref{growth-regime}.
Thus $R_\star$ is finite. Also, \eqref{Rstar-explicit} gives $R>R_1$ and $R^{2/m}>R_1$; Lemma~\ref{le4.1} then proves $\mathcal H\subset S_R^2$.

For a state outside $S_R$, \eqref{energy-tail-bound}, $G\ge a$, and \eqref{radius-45} give
\begin{equation}\label{Pl}
F(x,v)G(x,v)\ge a e^{\bigl(b_0/4\wedge c_0\bigr)R^2-B_0/4}\ge4C_5,
\qquad (x,v)\notin S_R,\quad R\ge R_\star.
\end{equation}
If one state lies outside $S_R$, Lemma~\ref{le4.3} gives
\[
\frac{2C_5}{\eta_{3,R}}+F(x,v)+F(x',v')
\le3\bigl(F(x,v)\vee F(x',v')\bigr).
\]
Combining this inequality with Lemma~\ref{le4.2} yields
\begin{equation}\label{ph}
\begin{aligned}
&\left(\frac{2C_5}{\eta_{3,R}}+F(x,v)+F(x',v')\right)
\left(\gamma\bigl(A_R+\alpha_R-1\bigr)+1+
\frac{|\nabla U(x)-\nabla U(x')|}{\gamma|x-x'|}\right)\\
\le{}&\frac12\bigl(F(x,v)\vee F(x',v')\bigr)
\bigl(G(x,v)\vee G(x',v')\bigr).
\end{aligned}
\end{equation}
Both $F=e^{\mathcal V}$ and $G=a(1+\mathcal V)$ are increasing functions of $\mathcal V$.
Their maxima are therefore attained at the same state, and
\begin{equation}\label{sop}
\bigl(F(x,v)\vee F(x',v')\bigr)\bigl(G(x,v)\vee G(x',v')\bigr)
\le F(x,v)G(x,v)+F(x',v')G(x',v').
\end{equation}
Equations~\eqref{ph}--\eqref{sop} bound the product on the left of \eqref{eqpr42} by one half of $FG+F'G'$.
Equation~\eqref{Pl} makes the remaining half at least $2C_5$, proving \eqref{eqpr42}.
\end{proof}

We finish with the deterministic slope bounds used in the small shifted
difference region. Concavity and \(f_R(0)=0\) imply \(0\le r f'_{R,-}(r)\le f_R(r)\le r\) for \(r\ge0\).
The defining integral and \(g_R\ge1/2\) also give
\begin{equation}\label{fprime-lower}
\begin{aligned}
f'_{R,-}(r)\ge\tfrac12\phi_R(\eta_{2,R}),~0\le r\le\eta_{2,R},\qquad f'_{R,-}(r)\ge\tfrac12\phi_R(\eta_{4,R}),~0\le r\le\eta_{4,R}.
\end{aligned}
\end{equation}
These include the positive left derivative at the outer cutoff.
Since \(\phi_R(\eta_{4,R})\le\phi_R(\eta_{2,R})\),
\begin{equation}\label{eta35-budget}
\eta_{3,R}+\eta_{5,R}\le\frac{\gamma}{16}\phi_R(\eta_{2,R}).
\end{equation}
Consequently,
\begin{equation}\label{core-drift-budget}
-\gamma r f'_{R,-}(r)+(\eta_{3,R}+\eta_{5,R})f_R(r)
\le-\frac{7\gamma}{16}\phi_R(\eta_{2,R})r,
\qquad 0\le r\le\eta_{2,R},
\end{equation}
and
\begin{equation}\label{outer-drift-budget}
-\gamma r f'_{R,-}(r)+\eta_{5,R}f_R(r)
\le-\frac{15\gamma}{32}\phi_R(\eta_{4,R})r,
\qquad 0\le r\le\eta_{4,R}.
\end{equation}
For \eqref{core-drift-budget}, use the first bound in
\eqref{fprime-lower}, \eqref{eta35-budget}, and \(f_R(r)\le r\);
the second estimate follows from the second bound and
\(\eta_{5,R}\le\gamma\phi_R(\eta_{4,R})/32\).
If \(r>\eta_{2,R}\), \eqref{eta2} implies
\(F(x,v)G(x,v)+F(x',v')G(x',v')>2C_5\), so the Lyapunov
remainder is nonpositive there.

\subsection{Regularized coupling estimates}

We now estimate the weighted distance along the regularized couplings of Section~\ref{Sec3}.

For each \(n\ge2\), set
\[Z_t^{(n)}:=X_t^{(n)}-X_t^{\prime(n)},\qquad
Y_t^{(n)}:=V_t^{(n)}-V_t^{\prime(n)},
\qquad
Q_t^{(n)}:=\alpha_R Z_t^{(n)}+\gamma^{-1}Y_t^{(n)}.\]
Set
\begin{align}
r_t^{(n)}
&:=r\left((X_t^{(n)},V_t^{(n)}),
(X_t^{\prime(n)},V_t^{\prime(n)})\right)
=A_R|Z_t^{(n)}|+|Q_t^{(n)}|,\quad \text{and}
\notag\\
\rho_t^{(n)}
&:=\rho_R\left((X_t^{(n)},V_t^{(n)}),
(X_t^{\prime(n)},V_t^{\prime(n)})\right)
=f(r_t^{(n)})\mathcal E_t^{(n)},\quad \text{where}
\notag\\
\mathcal E_t^{(n)}
&:=1+\varepsilon_R F(X_t^{(n)},V_t^{(n)})
+\varepsilon_R F(X_t^{\prime(n)},V_t^{\prime(n)}).
\label{ETn}
\end{align}

\begin{lemma}\label{Le1n}
Suppose that Assumptions \ref{as4.0}--\ref{as4.3} hold. Fix \(R\ge R_\star\),
\(\delta\in(0,1)\), and \(\beta\in(0,2)\). Let \(u_{\rm r}\) and \(u_{\rm s}\)
satisfy \eqref{rc}, and let the regularized coupling be given by
\eqref{regu-coup}.
For each $n\ge2$, there are a continuous local martingale $M^{\rho,(n)}$ and a continuous nondecreasing process $K^{\rho,(n)}$, both starting from zero, such that
\[
\rho_t^{(n)}=\rho_0^{(n)}+\int_0^tH_s^{(n)}\,\mathrm ds+M_t^{\rho,(n)}-K_t^{\rho,(n)},\qquad t\ge0.
\]
In particular, $\mathrm d\rho_t^{(n)}\le H_t^{(n)}\,\mathrm dt+\mathrm dM_t^{\rho,(n)}$ in the sense of semimartingale differentials. The drift upper bound is
\begin{equation}\label{Hn}
\begin{aligned}
H_t^{(n)}:=&
\mathcal E_t^{(n)}f'_-(r_t^{(n)})
\Big[-\gamma\alpha_RA_R|Z_t^{(n)}|
+\gamma\bigl(A_R+\alpha_R-1\bigr)|Q_t^{(n)}|-\gamma\alpha_R(\alpha_R-1)\mathbf e(Q_t^{(n)})^{\mathrm T}Z_t^{(n)}\\
&\hspace{2.5cm}-\gamma^{-1}\mathbf e(Q_t^{(n)})^{\mathrm T}\bigl(\nabla U(X_t^{(n)})-\nabla U(X_t^{\prime(n)})
\bigr)\Big]\\
&+\varepsilon_R f(r_t^{(n)})
\Bigl[2C_5-F(X_t^{(n)},V_t^{(n)})G(X_t^{(n)},V_t^{(n)})-F(X_t^{\prime(n)},V_t^{\prime(n)})G(X_t^{\prime(n)},V_t^{\prime(n)})\Bigr]\\
&+4\gamma^{-1}\theta_{n,\beta}\mathcal E_t^{(n)}(u_{{\rm r},t}^{(n)})^2
|\mathbf e_n(Q_t^{(n)})|^2\left[
f''(r_t^{(n)})+\gamma\eta_{1,R}f'_-(r_t^{(n)})\right].
\end{aligned}
\end{equation}
\end{lemma}
\begin{proof}
Fix $n\ge2$ and let $j$ exceed the norm of the initial pair. Stop the joint process at $\tau_j^{(n)}$, its exit time from the ball of radius $j$ in $E^2$. Theorem~\ref{thwcc} gives $\tau_j^{(n)}\uparrow\infty$ almost surely. All limiting calculations below are first made on $[0,T\wedge\tau_j^{(n)}]$ for fixed $T,j$, and then extended by localization.

\medskip
\noindent
\textbf{Step 1.} Equations for \(Z^{(n)}\), \(Y^{(n)}\), and \(Q^{(n)}\).
Subtracting the two position equations in \eqref{regu-coup} gives \(\d Z_t^{(n)}=Y_t^{(n)}\,\d t\).
Since \(Y_t^{(n)}=\gamma Q_t^{(n)}
-\gamma\alpha_R Z_t^{(n)}\),
we obtain
\begin{equation}\label{Z-eq}
\d Z_t^{(n)}=\left[\gamma Q_t^{(n)}-\gamma\alpha_R Z_t^{(n)}\right]\d t.
\end{equation}
Subtracting the two velocity equations gives
\begin{equation}\label{Y-eq}
\begin{aligned}
\d Y_t^{(n)}=&-\gamma Y_t^{(n)}\,\d t
-\bigl(\nabla U(X_t^{(n)})-\nabla U(X_t^{\prime(n)})\bigr)\,\d t\\
&+\sqrt{2\gamma}\,u_{{\rm r},t}^{(n)}
\Big(\mathbf I_d- \mathsf R_{n,\beta}(Q_t^{(n)})\Big)\d B_t^{1,n}-\sqrt{2\gamma}\,u_{{\rm r},t}^{(n)}\Sigma_{n,\beta}(Q_t^{(n)})\,\d B_t^{3,n}.
\end{aligned}
\end{equation}
Combining \eqref{Z-eq} and \eqref{Y-eq} with $Q^{(n)}=\alpha_R Z^{(n)}+\gamma^{-1}Y^{(n)}$ gives
\begin{equation}\label{Q-eq}
\begin{aligned}
\d Q_t^{(n)}
=&\Big[\gamma(\alpha_R-1)Q_t^{(n)}
-\gamma\alpha_R(\alpha_R-1)Z_t^{(n)}-\gamma^{-1}
\bigl(\nabla U(X_t^{(n)})-\nabla U(X_t^{\prime(n)})\bigr)\Big]\d t\\
&+\sqrt{2\gamma^{-1}}u_{{\rm r},t}^{(n)}\left(\mathbf I_d- \mathsf R_{n,\beta}(Q_t^{(n)})\right)\d B_t^{1,n}-\sqrt{2\gamma^{-1}}u_{{\rm r},t}^{(n)}\Sigma_{n,\beta}(Q_t^{(n)})\d B_t^{3,n}.
\end{aligned}
\end{equation}

\medskip
\noindent
\textbf{Step 2.} Equations for \(|Z^{(n)}|\) and \(|Q^{(n)}|\). The process \(Z^{(n)}\) is absolutely continuous. The chain rule for absolutely continuous vector-valued paths, equivalently the smooth approximation \(z\mapsto\sqrt{|z|^2+\epsilon^2}\), followed by dominated convergence as \(\epsilon\downarrow0\), yields, for Lebesgue-a.e.
\(t\geq0\),
\(\d |Z_t^{(n)}| =\mathbf e(Z_t^{(n)})^{\mathrm T}\d Z_t^{(n)}.\)
Using \eqref{Z-eq}, we obtain
\begin{equation}\label{Z-abs}
\begin{aligned}
\d |Z_t^{(n)}|
&=\left[-\gamma\alpha_R|Z_t^{(n)}|
+\gamma\mathbf e(Z_t^{(n)})^{\mathrm T}Q_t^{(n)}
\right]\d t\\
&=\left[-\gamma\alpha_R|Z_t^{(n)}|+\gamma|Q_t^{(n)}|\right]\d t-\d K_t^{Z,(n)}.
\end{aligned}
\end{equation}
Here
\(K_t^{Z,(n)}:=\gamma\int_0^t\left[ |Q_s^{(n)}|-\mathbf e(Z_s^{(n)})^{\mathrm T}Q_s^{(n)}\right]\d s\)
is continuous and nondecreasing.

Since \(\mathsf R_{n,\beta}\) is symmetric, \eqref{AAI} implies
\begin{equation}\label{cov-id}
\left(\mathbf I_d- \mathsf R_{n,\beta}(q)\right)
\left(\mathbf I_d-\mathsf R_{n,\beta}(q)\right)^{\mathrm T}
+ \Sigma_{n,\beta}(q)\Sigma_{n,\beta}(q)^{\mathrm T}
=4\theta_{n,\beta}\mathbf e_n(q)\mathbf e_n(q)^{\mathrm T}.
\end{equation}
Consequently, the matrix-valued quadratic variation of \(Q^{(n)}\) is
\begin{equation}\label{Q-qv}
\d\langle Q^{(n)}\rangle_t=8\gamma^{-1}\theta_{n,\beta}(u_{{\rm r},t}^{(n)})^2\mathbf e_n(Q_t^{(n)})\mathbf e_n(Q_t^{(n)})^{\mathrm T}\d t.
\end{equation}
For \(\varepsilon>0\), set \(\psi_\varepsilon(q):=\sqrt{|q|^2+\varepsilon^2}\). For \(q\neq \mathbf 0\), \(\mathbf e_n(q)=\chi_n(|q|)\mathbf e(q)\), and hence \(\langle q, \mathbf e_n(q)\rangle^2
=|q|^2|\mathbf e_n(q)|^2\).
It follows that
\begin{equation}\label{hessian}
\operatorname{tr}
\left[\nabla^2\psi_\varepsilon(q)\mathbf e_n(q)\mathbf e_n(q)^{\mathrm T}\right]=\frac{|\mathbf e_n(q)|^2}{\sqrt{|q|^2+\varepsilon^2}}
-\frac{\langle q, \mathbf e_n(q)\rangle^2}
{(|q|^2+\varepsilon^2)^{\frac{3}{2}}}=\frac{\varepsilon^2|\mathbf e_n(q)|^2}{(|q|^2+\varepsilon^2)^{\frac{3}{2}}}.
\end{equation}
The same equality holds at \(q= \mathbf 0\), because \(\mathbf e_n(\mathbf 0)=\mathbf 0\).
Furthermore, \(\mathbf e_n(q)=\mathbf 0\) whenever \(|q|\leq n^{-1}\). Thus
\[0\leq\frac{\varepsilon^2|\mathbf e_n(q)|^2} {(|q|^2+\varepsilon^2)^{\frac{3}{2}}}\leq\varepsilon^2n^3,\qquad q\in\mathbb R^d.\]
Apply It\^o's formula to \(\psi_\varepsilon(Q_t^{(n)})\), first after stopping
the processes in a compact set. By \eqref{Q-qv} and \eqref{hessian}, the
second-order term on \([0,T]\) is bounded by
\[\frac{4\theta_{n,\beta}}{\gamma}
\int_0^T(u_{{\rm r},s}^{(n)})^2
\frac{\varepsilon^2|\mathbf e_n(Q_s^{(n)})|^2}
{(|Q_s^{(n)}|^2+\varepsilon^2)^{\frac{3}{2}}}
\d s\leq\frac{4\theta_{n,\beta}}{\gamma}T\varepsilon^2n^3,\]
which tends to zero as \(\varepsilon\downarrow0\), for fixed \(n\). The drift
terms converge by dominated convergence because \(|\nabla\psi_\varepsilon|\le1\).
For the martingale terms, the difference between the regularized integrand $\nabla\psi_\varepsilon(Q^{(n)})$ and $\mathbf e(Q^{(n)})$ tends pointwise to zero after multiplication by the diffusion matrices in \eqref{Q-eq}. At $Q^{(n)}=0$ those matrices vanish. The integrands are bounded, so dominated convergence and the It\^o isometry give convergence in $L^2$ at each stopped time; Doob's maximal inequality gives convergence uniformly on $[0,T]$ in $L^2$ after stopping. Letting $\varepsilon\downarrow0$ first, and then $j\to\infty$, yields
\begin{equation}\label{Q-abs}
\begin{aligned}
\d |Q_t^{(n)}|=\Big[\gamma(\alpha_R-1)|Q_t^{(n)}|
&-\gamma\alpha_R(\alpha_R-1)\mathbf e(Q_t^{(n)})^{\mathrm T}Z_t^{(n)}
\\
&-\gamma^{-1}\mathbf e(Q_t^{(n)})^{\mathrm T}
\left(\nabla U(X_t^{(n)})-\nabla U(X_t^{\prime(n)})\right)
\Big]\d t+\d M_t^{Q,(n)},
\end{aligned}
\end{equation}
where
\begin{equation*}
\d M_t^{Q,(n)}=\sqrt{2\gamma^{-1}}u_{{\rm r},t}^{(n)}
\mathbf e(Q_t^{(n)})^{\mathrm T}
\left(\mathbf I_d-\mathsf R_{n,\beta}(Q_t^{(n)})\right)\d B_t^{1,n}-\sqrt{2\gamma^{-1}}u_{{\rm r},t}^{(n)}
\mathbf e(Q_t^{(n)})^{\mathrm T}\Sigma_{n,\beta}(Q_t^{(n)})\d B_t^{3,n}.
\end{equation*}
By the independence of \(B^{1,n}\) and \(B^{3,n}\),
\eqref{cov-id} gives
\begin{equation*}
\d\langle M^{Q,(n)}\rangle_t=8\gamma^{-1}\theta_{n,\beta}(u_{{\rm r},t}^{(n)})^2
|\mathbf e_n(Q_t^{(n)})|^2\d t.
\end{equation*}

\medskip
\noindent
\textbf{Step 3.} The covariation between \(f(r^{(n)})\) and \(\mathcal E^{(n)}\).
Multiplying \eqref{Z-abs} by \(A_R\) and adding \eqref{Q-abs}, we obtain
the exact semimartingale decomposition
\begin{equation}\label{r-eq}
\begin{aligned}
\d r_t^{(n)}=&
\Big[\!-\gamma\alpha_RA_R|Z_t^{(n)}|
+\gamma\bigl(A_R+\alpha_R-1\bigr)|Q_t^{(n)}|-\gamma\alpha_R(\alpha_R-1)\mathbf e(Q_t^{(n)})^{\mathrm T}Z_t^{(n)}\\
&-\gamma^{-1}\mathbf e(Q_t^{(n)})^{\mathrm T}
\left(\nabla U(X_t^{(n)})-\nabla U(X_t^{\prime(n)})\right)\!\Big]\d t+\d M_t^{Q,(n)}
-A_R\,\d K_t^{Z,(n)}.
\end{aligned}
\end{equation}
Since \(K_t^{Z,(n)}\) and \(|Z^{(n)}|\) are of finite variation, the local martingale part of \(r^{(n)}\) is \(M^{Q,(n)}\). Therefore,
\begin{equation}\label{r-qv}
\d\langle r^{(n)}\rangle_t=
8\gamma^{-1}\theta_{n,\beta}(u_{{\rm r},t}^{(n)})^2
|\mathbf e_n(Q_t^{(n)})|^2\d t.
\end{equation}
Let \(\mu_f\) be the distributional second derivative of the concave
extension of \(f\). Choose local time \(L_t^a(r^{(n)})\) right-continuous in \(a\), with normalization
\begin{equation}\label{local-time-normalization}
 \int h(a)L_t^a(r^{(n)})\,\mathrm da
 =\int_0^t h(r_s^{(n)})\,\mathrm d\langle r^{(n)}\rangle_s
\end{equation}
for nonnegative Borel \(h\).
The local time in \eqref{local-time-normalization} is twice that used in
\cite[Chapter~3, Theorem~7.1(v) and Corollary~7.2, pp.~218--219]{KS1991}.
In this normalization, the generalized It\^o--Tanaka formula gives
\[
 f(r_t^{(n)})=f(r_0^{(n)})+\int_0^t f'_-(r_s^{(n)})\d r_s^{(n)}
 +\frac12\int_{\mathbb R}L_t^a(r^{(n)})\mu_f(\d a).
\]
On the intervals
\((0,\eta_{2,R})\) and \((\eta_{2,R},\eta_{4,R})\),
\(\mu_f(\d a)=f''(a)\d a\). Since local time is non-negative, \eqref{R1} and the occupation-time formula \cite[Chapter~3, Theorem~7.1(iii), p.~218]{KS1991} yield
\[\frac{1}{2}\int_{\mathbb R}L_t^a(r^{(n)})\mu_f(\d a)
\leq\frac{1}{2}\int_0^t f''(r_s^{(n)})\d\langle r^{(n)}\rangle_s.\]
The only atom of \(\mu_f\) is
\(-f'_-(\eta_{4,R})\delta_{\eta_{4,R}}\): the extension is \(f(r)=r\)
for \(r<0\), with matching derivative at zero, and \(f'\) is continuous
at \(\eta_{2,R}\). Its contribution is
\(-\tfrac12 f'_-(\eta_{4,R})L_t^{\eta_{4,R}}(r^{(n)})\le0\).
Thus the formula retains the positive left derivative at $\eta_{4,R}$.
Moreover, by the same occupation-time formula, we have
\[\int_0^t\1_{\{r_s^{(n)}\in\{0,\eta_{2,R},\eta_{4,R}\}\}} \d\langle r^{(n)}\rangle_s=0.\]
By \eqref{r-qv}, this also says that
\(\theta_{n,\beta}(u_{{\rm r},s}^{(n)})^2
|\mathbf e_n(Q_s^{(n)})|^2\) vanishes for \(\mathrm ds\)-almost every
time at any of these levels. Thus the assigned values of \(f''\) there
are immaterial. The diffusion and covariation estimates below therefore hold for $\mathrm dt\otimes\mathbb P$-almost every $(t,\omega)$.
Since $f'_-\ge0$, the process
\[
K_t^{f,(n)}:=A_R\int_0^t f'_-(r_s^{(n)})\,\mathrm dK_s^{Z,(n)}
+\tfrac12 f'_-(\eta_{4,R})L_t^{\eta_{4,R}}(r^{(n)})
\]
is continuous and nondecreasing. It collects the two nonpositive terms in the It\^o--Tanaka formula. Dropping $-\mathrm dK_t^{f,(n)}$ and using \eqref{r-eq} and \eqref{r-qv}, we obtain
\begin{equation}\label{f-r}
\begin{aligned}
\d f(r_t^{(n)})\leq f'_-(r_t^{(n)})\Big[\!&-\gamma\alpha_RA_R|Z_t^{(n)}|+\gamma\bigl(A_R+\alpha_R-1\bigr)|Q_t^{(n)}|-\gamma\alpha_R(\alpha_R-1)\mathbf e(Q_t^{(n)})^{\mathrm T}Z_t^{(n)}\\
&-\gamma^{-1}\mathbf e(Q_t^{(n)})^{\mathrm T}
\left(\nabla U(X_t^{(n)})-\nabla U(X_t^{\prime(n)})
\right)\Big]\d t\\
&+4\gamma^{-1}\theta_{n,\beta}(u_{{\rm r},t}^{(n)})^2|\mathbf e_n(Q_t^{(n)})|^2
f''(r_t^{(n)})\d t+f'_-(r_t^{(n)})\d M_t^{Q,(n)}.
\end{aligned}
\end{equation}

Next we estimate the evolution of \(\mathcal E_t^{(n)}\). By It\^o's formula applied to \(F(X_t^{(n)},V_t^{(n)})\) and \(F(X_t^{\prime(n)},V_t^{\prime(n)})\), and by the marginal property guaranteed by Theorem \ref{thwcc}, both coordinate processes have generator \(\mathcal L\) defined in \eqref{eq4.3}. Therefore
\begin{equation*}
\d \mathcal E_t^{(n)}=\varepsilon_R(\mathcal LF)(X_t^{(n)},V_t^{(n)})\d t+
\varepsilon_R(\mathcal LF)(X_t^{\prime(n)},V_t^{\prime(n)})\d t+\d\widetilde M_t^{(n)},
\end{equation*}
where
\begin{equation*}
\begin{aligned}
&\d\widetilde M_t^{(n)}=\varepsilon_R\sqrt{2\gamma}
\nabla_vF(X_t^{(n)},V_t^{(n)})^{\mathrm T}
\left[u_{{\rm r},t}^{(n)}\,\d B_t^{1,n}
+u_{{\rm s},t}^{(n)}\,\d B_t^{2,n}\right]\\
&\!+\!\varepsilon_R\sqrt{2\gamma}\nabla_vF(\!X_t^{\prime(n)},V_t^{\prime(n)})^{\mathrm T}\Bigl[\!u_{{\rm r},t}^{(n)}\mathsf R_{n,\beta}(\!Q_t^{(n)})\d B_t^{1,n}\!+\!u_{{\rm s},t}^{(n)}\d B_t^{2,n}\!+\!u_{{\rm r},t}^{(n)} \Sigma_{n,\beta}(\!Q_t^{(n)})\d B_t^{3,n}\!\Bigr].
\end{aligned}
\end{equation*}
Proposition \ref{pr4.1} gives
\begin{equation}\label{E-bound}
\begin{aligned}
\d \mathcal E_t^{(n)}\leq{}&\varepsilon_R
\left[2C_5-F(X_t^{(n)},V_t^{(n)})G(X_t^{(n)},V_t^{(n)})\right.\\
&\left.\hspace{26mm}-F(X_t^{\prime(n)},V_t^{\prime(n)})G(X_t^{\prime(n)},V_t^{\prime(n)})\right]\d t+\d\widetilde M_t^{(n)}.
\end{aligned}
\end{equation}
Only terms driven by the same Brownian motion contribute to the covariation. The $B^{2,n}$ component is absent from $M^{Q,(n)}$, while independence eliminates cross terms between $B^{1,n}$ and $B^{3,n}$. For this calculation, abbreviate
\[
\begin{aligned}
\mathsf R&=\mathsf R_{n,\beta}(Q_t^{(n)}),\qquad
\Sigma=\Sigma_{n,\beta}(Q_t^{(n)}),\qquad e=\mathbf e(Q_t^{(n)}),\\
p&=\nabla_vF(X_t^{(n)},V_t^{(n)}),\qquad
p'=\nabla_vF(X_t^{\prime(n)},V_t^{\prime(n)}).
\end{aligned}
\]
The $B^{1,n}$ and $B^{3,n}$ contributions to the covariation density are, respectively,
\[
2\varepsilon_R(u_{{\rm r},t}^{(n)})^2e^{\mathrm T}(\mathbf I_d-\mathsf R)(p+\mathsf R^{\mathrm T}p'),
\qquad
-2\varepsilon_R(u_{{\rm r},t}^{(n)})^2e^{\mathrm T}\Sigma\Sigma^{\mathrm T}p'.
\]
Symmetry of $\mathsf R$ and \eqref{AAI} give \((\mathbf I_d-\mathsf R)\mathsf R^{\mathrm T}-\Sigma\Sigma^{\mathrm T}=-(\mathbf I_d-\mathsf R)\).
Adding the two contributions therefore yields
\[
\begin{aligned}
\mathrm d\langle M^{Q,(n)},\widetilde M^{(n)}\rangle_t
&=2\varepsilon_R(u_{{\rm r},t}^{(n)})^2e^{\mathrm T}
 \bigl[(\mathbf I_d-\mathsf R)p+
 ((\mathbf I_d-\mathsf R)\mathsf R^{\mathrm T}-\Sigma\Sigma^{\mathrm T})p'\bigr]\,\mathrm dt\\
&=2\varepsilon_R(u_{{\rm r},t}^{(n)})^2e^{\mathrm T}(\mathbf I_d-\mathsf R)(p-p')\,\mathrm dt.
\end{aligned}
\]
Moreover, $\mathbf I_d-\mathsf R=2\theta_{n,\beta}\mathbf e_n(Q_t^{(n)})\mathbf e_n(Q_t^{(n)})^{\mathrm T}$. Since $\mathbf e_n$ is parallel to $\mathbf e$ away from zero and both vanish at zero,
$e^{\mathrm T}(\mathbf I_d-\mathsf R)=2\theta_{n,\beta}|\mathbf e_n(Q_t^{(n)})|^2e^{\mathrm T}$.
Returning to the full notation, we obtain
\begin{equation}\label{cov-exact}
\begin{aligned}
\mathrm d\langle M^{Q,(n)},\widetilde M^{(n)}\rangle_t
={}&4\varepsilon_R\theta_{n,\beta}(u_{{\rm r},t}^{(n)})^2
|\mathbf e_n(Q_t^{(n)})|^2\mathbf e(Q_t^{(n)})^{\mathrm T}\\
&\quad\times[\nabla_vF(X_t^{(n)},V_t^{(n)})
-\nabla_vF(X_t^{\prime(n)},V_t^{\prime(n)})]\,\mathrm dt.
\end{aligned}
\end{equation}
If \(u_{{\rm r},t}^{(n)}=0\), the right-hand side is zero. If \(u_{{\rm r},t}^{(n)}>0\), then by the choice of \(u_{\rm r}\) in \eqref{rc}, both
\((X_t^{(n)},V_t^{(n)})\) and \((X_t^{\prime(n)},V_t^{\prime(n)})\) belong to \(S_R\).
By \eqref{velocity-weight-bound} and \eqref{ETn},
\[\varepsilon_R\left|\nabla_vF(X_t^{(n)},V_t^{(n)})
-\nabla_vF(X_t^{\prime(n)},V_t^{\prime(n)})\right|
\leq \mathcal E_t^{(n)}\eta_{1,R}.\]
Combining this and \eqref{cov-exact}, and using \(\d\langle f(r^{(n)}),\mathcal E^{(n)}\rangle_t
=f'_-(r_t^{(n)})\d\langle M^{Q,(n)},\widetilde M^{(n)}\rangle_t\),
we find
\begin{equation}\label{cov-final}
\d\langle f(r^{(n)}),\mathcal E^{(n)}\rangle_t
\leq 4\theta_{n,\beta}(u_{{\rm r},t}^{(n)})^2
|\mathbf e_n(Q_t^{(n)})|^2\mathcal E_t^{(n)}\eta_{1,R}
f'_-(r_t^{(n)})\d t.
\end{equation}

\medskip
\noindent
\textbf{Step 4.} The product formula.
By the product formula for continuous semimartingales,
\[\d\rho_t^{(n)}=\mathcal E_t^{(n)}\,\d f(r_t^{(n)})+f(r_t^{(n)})\,\d \mathcal E_t^{(n)}+\d\langle f(r^{(n)}),\mathcal E^{(n)}\rangle_t.\]
Substituting \eqref{f-r}, \eqref{E-bound}, and
\eqref{cov-final} yields exactly the drift
\(H_t^{(n)}\) in \eqref{Hn}. The remaining stochastic terms are
\[M_t^{\rho,(n)}:=
\int_0^t\mathcal E_s^{(n)}f'_-(r_s^{(n)})\d M_s^{Q,(n)}
+\int_0^tf(r_s^{(n)})\d\widetilde M_s^{(n)}.\]
On $[0,T\wedge\tau_j^{(n)}]$, the weight, its velocity gradient, $f$, and $f'_-$ are bounded, as are the diffusion matrices. Thus both stochastic integrals are square-integrable after stopping, and $M^{\rho,(n)}$ is a continuous local martingale. The omitted term $\int_0^t\mathcal E_s^{(n)}\,\mathrm dK_s^{f,(n)}$ is nondecreasing. The differences introduced by the Lyapunov upper bound \eqref{E-bound} and the covariation upper bound \eqref{cov-final} are also integrals of nonnegative densities, since $f,f'_-\ge0$. Their sum defines the continuous nondecreasing process $K^{\rho,(n)}$ in the statement. This proves the differential inequality with its asserted finite-variation meaning.
\end{proof}

The next estimate is uniform in the initial pair. Its error constants are computed after $R$ has been fixed.

\begin{lemma}\label{lemrde}
Suppose that Assumptions \ref{as4.0}--\ref{as4.3} hold. Fix \(R\ge R_\star\),
\(\delta\in(0,1)\), \(\beta\in(0,2)\), and let \(c_{n,\beta}:=1\wedge(\theta_{n,\beta}\eta_{5,R})\).
There is a finite constant \(C_R^{\rm err}>0\), independent of
\(n,\delta\) and the initial pair, such that the drift in \eqref{Hn}
satisfies
\begin{equation}\label{Hn-final-bound}
 H_t^{(n)}\le
 -c_{n,\beta}\rho_t^{(n)}
 +C_R^{\rm err}(\delta+n^{-\beta}+n^{-1})
\end{equation}
for \(\mathrm dt\otimes\mathbb P\)-almost every \((t,\omega)\).
Consequently,
\begin{equation}\label{master-finite-n}
 \mathrm d\rho_t^{(n)}
 \le\left[-c_{n,\beta}\rho_t^{(n)}
       +C_R^{\rm err}(\delta+n^{-\beta}+n^{-1})\right]\mathrm dt
       +\mathrm dM_t^{\rho,(n)}
\end{equation}
in the sense of semimartingale differentials.
\end{lemma}

\begin{proof}
We first identify the deterministic bounds used in the four
regions. At a generic pair \(\xi=(x,v)\), \(\xi'=(x',v')\), write \(z=x-x'\), \(y=v-v'\) and \(q=\alpha_Rz+\gamma^{-1}y\).
The norm-drift expression appearing in \eqref{Hn} is
\begin{equation}\label{D-R-def}
 D_R(\xi,\xi')=
 -\gamma\alpha_RA_R|z|+\gamma\bigl(A_R+\alpha_R-1\bigr)|q|
 -\gamma\alpha_R(\alpha_R-1)\mathbf e(q)^{\mathrm T}z-\gamma^{-1}\mathbf e(q)^{\mathrm T}\bigl(\nabla U(x)-\nabla U(x')\bigr).
\end{equation}
On \(\overline S_R\times\overline S_R\), Assumption \ref{as4.1}
gives \(|\nabla U(x)-\nabla U(x')|\le L_R|z|\).
The definitions in \eqref{eq4.5} give
\[
 A_R=\alpha_R+\frac{L_R}{\gamma^2(\alpha_R-1)},\qquad
 -\gamma\alpha_RA_R+\gamma\alpha_R(\alpha_R-1)+L_R/\gamma=-\gamma A_R.
\]
Consequently, the calculation in \eqref{local-kinetic-drift} and the definition \eqref{rxv} give
\begin{align}
 D_R(\xi,\xi')
 &\le[-\gamma\alpha_RA_R+\gamma\alpha_R(\alpha_R-1)+L_R/\gamma]|z|
                  +\gamma\bigl(A_R+\alpha_R-1\bigr)|q|\notag\\
 &=-\gamma A_R|z|+\gamma\bigl(A_R+\alpha_R-1\bigr)|q|
   =-\gamma r+\gamma(A_R+\alpha_R)|q|,\label{small-q-drift}\\
 D_R(\xi,\xi')&\le\gamma\bigl(A_R+\alpha_R-1\bigr)r.\label{D-R-linear}
\end{align}
On the whole state space, $A_R\ge\alpha_R-1$ also gives
\begin{equation}\label{B-R-def}
 D_R(\xi,\xi')\le B_R(\xi,\xi')
 :=\gamma\bigl(A_R+\alpha_R-1\bigr)|q|+\gamma^{-1}|\nabla U(x)-\nabla U(x')|,
\end{equation}
because the sum of the two terms involving \(z\) in \eqref{D-R-def}
is bounded above by $-\gamma\alpha_R[A_R-(\alpha_R-1)]|z|\le0$.

For the four cases, suppress the time and $n$ superscripts in state functions, and write
\[
\mathcal E(\xi,\xi'):=1+\varepsilon_RF(\xi)+\varepsilon_RF(\xi'),\qquad
\rho_R(\xi,\xi')=f(r)\mathcal E(\xi,\xi').
\]
By \eqref{Hn}, the diffusion and covariation terms combine as
\begin{equation*}
 K_n=\frac4\gamma\theta_{n,\beta}|\mathbf e_n(q)|^2u_{\rm r}^2\mathcal E(\xi,\xi')
       \bigl[f''(r)+\gamma\eta_{1,R}f'_-(r)\bigr].
\end{equation*}
Equation \eqref{fe} shows \(K_n\le0\) away from the two profile
levels; outside the cap both derivatives vanish.
At either level the multiplying factor
\(\theta_{n,\beta}|\mathbf e_n(q)|^2u_{\rm r}^2\) vanishes at almost every time by \eqref{r-qv}
and the occupation identity proved in Lemma \ref{Le1n}.
Thus \(K_n\le0\) in every time-integrated estimate below.

If \(\bigl(F(\xi)G(\xi)+F(\xi')G(\xi')\bigr)\le2C_5\), the lower bound \eqref{V1} implies
\(|\xi|,|\xi'|\le R_1\), since \(FG\ge a e^{c_0|\xi|^2}\).
The definition of \(\eta_{2,R}\) then implies \(r\le\eta_{2,R}\).
Using \(\varepsilon_R=\eta_{3,R}/(2C_5)\) and \(\mathcal E\ge1\),
we obtain
\begin{equation}\label{positive-core-bound}
 \varepsilon_R f(r)\bigl(2C_5-F(\xi)G(\xi)-F(\xi')G(\xi')\bigr)
 \le\eta_{3,R}\rho_R(\xi,\xi')
                 \mathbf1_{\{0<r\le\eta_{2,R}\}}.
\end{equation}
At \(r=0\) the left-hand side is zero.

For pairs in \(\overline S_R^2\), one has \(\mathcal E\le1+2\varepsilon_R\max_{\overline S_R}F\) and \(r\le\eta_{4,R}=2(A_R+\alpha_R)R^{2/m}+2\gamma^{-1}R\).
The maximum is finite by continuity of \(F\).
We consider the four regions determined by \eqref{rc}.

\medskip\noindent
\textbf{Case I: at least one state lies outside \(S_R\).}
Here \(u_{\rm r}=0\), so \(K_n=0\).
If \(r=0\), all the terms in \eqref{Hn} are zero, and the claim follows.
For \(r>0\), \eqref{B-R-def}, \(r\ge|q|\), and
\(r\ge A_R|z|\ge|z|\) yield
\[
 \frac{B_R(\xi,\xi')}{r}+1
 \le\gamma\bigl(A_R+\alpha_R-1\bigr)+1+\frac{|\nabla U(x)-\nabla U(x')|}{\gamma|z|},
\]
where the last quotient is defined as zero if \(z=0\).
Proposition \ref{pr4.2}, specifically \eqref{eqpr42}, therefore gives
\[
 \varepsilon_R f(r)\bigl(2C_5-F(\xi)G(\xi)-F(\xi')G(\xi')\bigr)
 \le-\mathcal E(\xi,\xi') f(r)\left(\frac{B_R(\xi,\xi')}{r}+1\right).
\]
Since \(rf'_-(r)\le f(r)\) by concavity, substitution in \eqref{Hn}
gives
\begin{equation*}
 H^{(n)}\le
 \mathcal E B_R\left(f'_-(r)-\frac{f(r)}r\right)-\rho_R
 \le-\rho_R.
\end{equation*}
The boundary \(\partial S_R\) is included in this case because \(S_R\)
is open.

\medskip\noindent
\textbf{Case II: both states lie in \(S_R\), and at least one lies
outside \(S_{R,\delta}\).}
On \(\overline S_R^2\), define the upper envelope
\begin{equation*}
 \mathscr J_R(\xi,\xi')
   :=\mathcal E(\xi,\xi')f'_-(r(\xi,\xi'))B_R(\xi,\xi')
        +\varepsilon_R f(r(\xi,\xi'))\bigl(2C_5-F(\xi)G(\xi)-F(\xi')G(\xi')\bigr).
\end{equation*}
Both \(\mathscr J_R\) and \(\rho_R\) are Lipschitz on this compact set.
Indeed, the difference maps are linear and norms are Lipschitz,
so \(r\) is Lipschitz with, for example, constant
\(\sqrt2(A_R+\alpha_R+\gamma^{-1})\) in the Euclidean norm on \(E^2\).
The force is Lipschitz on \(\{|x|\le R^{2/m}\}\) by Assumption~\ref{as4.1}, while \(F,G\) and their
first derivatives are bounded on \(\overline S_R\).
The restriction of \(f'_-\) to
\([0,\eta_{4,R}]\) is Lipschitz: formula \eqref{fe} bounds its
ordinary derivative on each of the two subintervals, $f'$
is continuous at $\eta_{2,R}$, and its value at \(\eta_{4,R}\)
is its left limit. Every distance attained in \(\overline S_R^2\)
belongs to this closed interval, so the right-side derivative jump at
the cap is never crossed. The elementary product and composition
inequalities for bounded Lipschitz functions now prove the assertion.

Let \(L_{\mathscr J,R}\) and \(L_{\rho,R}\) denote their Lipschitz
seminorms on \(\overline S_R^2\). Choose a state not in
\(S_{R,\delta}\). If its position component has norm in $[R^{2/m}-\delta,R^{2/m})$, extend that component radially to norm $R^{2/m}$.
Otherwise its velocity component has norm in $[R-\delta,R)$; extend that component radially to norm $R$.
The chosen component is nonzero because $R^{2/m},R>1>\delta$ by \eqref{Rstar-explicit}.
Leave all other components unchanged. The resulting pair
\((\widetilde\xi,\widetilde\xi')\) lies in \(\overline S_R^2\),
has a state on \(\partial S_R\), and is at distance at most \(\delta\)
from the original pair.
The calculation in Case I applies to the envelope itself at this pair:
\(\mathscr J_R(\widetilde\xi,\widetilde\xi')
 \le-\rho_R(\widetilde\xi,\widetilde\xi')\).
Using \(D_R\le B_R\) and \(K_n\le0\), we conclude that
\begin{align*}
 H^{(n)}
 \le\mathscr J_R(\xi,\xi')
 \le\mathscr J_R(\widetilde\xi,\widetilde\xi')
                         +L_{\mathscr J,R}\delta\le-\rho_R(\xi,\xi')+
                     (L_{\mathscr J,R}+L_{\rho,R})\delta.
\end{align*}
Set \(C_R^{\rm II}=L_{\mathscr J,R}+L_{\rho,R}<\infty\).

\medskip\noindent
\textbf{Case III: both states lie in \(S_{R,\delta}\), and
\(|\xi-\xi'|<\delta\).}
The elementary bounds \(|z|\le\delta\), \(|y|\le\delta\) give
\(r\le (A_R+\alpha_R+\gamma^{-1})\delta\). Consequently,
\(\rho_R\le \bigl(1+2\varepsilon_R\max_{\overline S_R}F\bigr)(A_R+\alpha_R+\gamma^{-1})\delta\).
Using \(f'_-\le1\), \eqref{D-R-linear},
\eqref{positive-core-bound}, and \(K_n\le0\) in \eqref{Hn},
we obtain \(H^{(n)}\le-\rho_R+C_R^{\rm III}\delta\), where
\[
 C_R^{\rm III}:=\bigl(1+2\varepsilon_R\max_{\overline S_R}F\bigr)
 (A_R+\alpha_R+\gamma^{-1})(\gamma\bigl(A_R+\alpha_R-1\bigr)+\eta_{3,R}+1).
\]

\medskip\noindent
\textbf{Case IV: both states lie in \(S_{R,\delta}\), and
\(|\xi-\xi'|\ge\delta\).}
Now \eqref{rc} gives \(u_{\rm r}=1\). We distinguish only whether
the direction regularization has reached its constant branch.

If $|q|\ge2/n$, then $\theta_{n,\beta}|\mathbf e_n(q)|^2=\theta_{n,\beta}$. Substitution of \eqref{fe} into \eqref{Hn}, followed by \eqref{positive-core-bound}, gives, away from the two profile levels,
\[
\begin{aligned}
H^{(n)}\le{}&\mathcal E(\xi,\xi') f'_-(r)[D_R(\xi,\xi')-\theta_{n,\beta}\gamma\bigl(A_R+\alpha_R-1\bigr)r]
-\theta_{n,\beta}\eta_{5,R}\mathcal E(\xi,\xi')\Phi(r)\\
&-\theta_{n,\beta}\eta_{3,R}\mathcal E(\xi,\xi')\Phi(r)\1_{\{0<r<\eta_{2,R}\}}
+\eta_{3,R}\rho_R\1_{\{0<r\le\eta_{2,R}\}}.
\end{aligned}
\]
Here the $\gamma\eta_{1,R}f'_-$ term in \eqref{Hn} cancels the corresponding term from $f''$. Using \eqref{D-R-linear} and $\Phi\ge f$, the remaining terms satisfy
\begin{align*}
H^{(n)}
&\le-\theta_{n,\beta}\eta_{5,R}\rho_R
+(1-\theta_{n,\beta})
 [\gamma\bigl(A_R+\alpha_R-1\bigr) r\mathcal E(\xi,\xi') f'_-(r)
 +\eta_{3,R}\rho_R\1_{\{0<r\le\eta_{2,R}\}}]\\
&\le-\theta_{n,\beta}\eta_{5,R}\rho_R+C_R^{\rm IV,1}n^{-\beta},
\end{align*}
where
\[
C_R^{\rm IV,1}:=
\bigl(1+2\varepsilon_R\max_{\overline S_R}F\bigr)
\eta_{4,R}(\gamma\bigl(A_R+\alpha_R-1\bigr)+\eta_{3,R}).
\]
At either profile level, the bracket density in \eqref{r-qv} is $8\gamma^{-1}\theta_{n,\beta}>0$ in this subcase. The occupation identity in Lemma~\ref{Le1n} therefore makes the exceptional times Lebesgue-null. Thus the estimate holds in the required time-integrated sense.

If \(|q|<2/n\), \eqref{small-q-drift} yields \(D_R(\xi,\xi')\le-\gamma r+2\gamma(A_R+\alpha_R)/n\).
Dropping $K_n\le0$ and using $\mathcal E(\xi,\xi') f'_-\le1+2\varepsilon_R\max_{\overline S_R}F$, we obtain
\[
\begin{aligned}
H^{(n)}\le&-\gamma r\mathcal E(\xi,\xi') f'_-(r)
+\varepsilon_Rf(r)[2C_5-F(\xi)G(\xi)-F(\xi')G(\xi')]\\
&+\frac{2\gamma(A_R+\alpha_R)(1+2\varepsilon_R\max_{\overline S_R}F)}{n}.
\end{aligned}
\]
For $r\le\eta_{2,R}$, \eqref{positive-core-bound} and \eqref{core-drift-budget} give
\[
\begin{aligned}
&-\gamma r\mathcal E(\xi,\xi') f'_-(r)
+\varepsilon_Rf(r)[2C_5-F(\xi)G(\xi)-F(\xi')G(\xi')]
+\eta_{5,R}\rho_R\\
\le&\mathcal E(\xi,\xi')[-\gamma r f'_-(r)+(\eta_{3,R}+\eta_{5,R})f(r)]\le0.
\end{aligned}
\]
For $\eta_{2,R}<r\le\eta_{4,R}$, the weight contribution is nonpositive by \eqref{eta2}. Equation~\eqref{outer-drift-budget} then gives the same conclusion with only $\eta_{5,R}f(r)$ in the bracket. These two ranges prove
\begin{equation}\label{case-IV-small-q}
H^{(n)}\le-\eta_{5,R}\rho_R+C_R^{\rm IV,2}n^{-1},\qquad
C_R^{\rm IV,2}:=2\gamma(A_R+\alpha_R)
\bigl(1+2\varepsilon_R\max_{\overline S_R}F\bigr).
\end{equation}
If \(r=0\), the same inequality holds because all terms in \(H^{(n)}\)
are zero. Thus \eqref{case-IV-small-q} includes the whole set
\(\{|q|<2/n\}\), including $q=0$ even when $b(\xi,\xi')\ne0$.

Finally, take
\(C_R^{\rm err}:= 1+C_R^{\rm II}+C_R^{\rm III} +C_R^{\rm IV,1}+C_R^{\rm IV,2}.\)
Cases I--III have rate one; Case IV has rate
\(\theta_{n,\beta}\eta_{5,R}\), with the two explicitly displayed
errors. Taking their minimum yields \eqref{Hn-final-bound}.
Lemma \ref{Le1n} then yields \eqref{master-finite-n}.
In particular, choosing \(\beta=1\) gives the deterministic error
\(O_R(\delta+n^{-1})\).
\end{proof}

\section{Proof of the main result}\label{Sec5}

We prove Theorem~\ref{Th} by integrating the regularized coupling estimate
and taking two weak limits on the space of continuous paths. The model
parameters, \(R\ge R_\star\), and \(\beta\in(0,2)\) remain fixed throughout.
We first let \(n\to\infty\) with \(\delta\) fixed, and then let
\(\delta\downarrow0\). Set \(c:=1\wedge\eta_{5,R}\).

\begin{proof}[Proof of Theorem~\ref{Th}]
\textbf{Step 1: a uniform estimate at finite regularization.}
Fix an initial pair \(\xi=(x,v)\), \(\xi'=(x',v')\), and
\(\delta\in(0,1)\). For \(n\ge2\), use the strong solution of
\eqref{regu-coup} with the switches in \eqref{rc}.
For \(j\ge1\), let
\[
 \tau_j^{(n)}
 =\inf\left\{s\ge0:
    |X_s^{(n)}|^2+|V_s^{(n)}|^2+
    |X_s^{\prime(n)}|^2+|V_s^{\prime(n)}|^2\ge j^2\right\}.
\]
For \(j\) larger than the initial norm, continuity keeps the stopped
process in a compact set, and Theorem~\ref{thwcc} gives
\(\tau_j^{(n)}\uparrow\infty\) almost surely. On
\([0,T\wedge\tau_j^{(n)}]\), the stochastic integrands in
\(M^{\rho,(n)}\) from Lemma~\ref{Le1n} are bounded: \(f,f'_-\),
the switches, and the diffusion matrices are bounded, while
\(F,\nabla_vF\) are bounded on the stopped state set. The stopped
stochastic integral, also after multiplication by
\(e^{c_{n,\beta}s}\), is therefore a square-integrable martingale
of mean zero.

Multiplying \eqref{master-finite-n} from Lemma~\ref{lemrde} by
\(e^{c_{n,\beta}t}\) and using the product rule gives
\[
 \mathrm d\bigl(e^{c_{n,\beta}t}\rho_t^{(n)}\bigr)
 \le C_R^{\rm err}(\delta+n^{-\beta}+n^{-1})
                              e^{c_{n,\beta}t}\,\mathrm dt
       +e^{c_{n,\beta}t}\,\mathrm dM_t^{\rho,(n)}.
\]
The discarded finite-variation measure in Lemma~\ref{Le1n} remains
nonpositive after multiplication by this positive integrating factor.
Taking expectations at \(t\wedge\tau_j^{(n)}\) yields
\begin{align*}
 \mathbb E\!\left[
   e^{c_{n,\beta}(t\wedge\tau_j^{(n)})}
       \rho_{t\wedge\tau_j^{(n)}}^{(n)}\right]
 \le\rho_R(\xi,\xi')+
 C_R^{\rm err}(\delta+n^{-\beta}+n^{-1})
     \int_0^t e^{c_{n,\beta}s}\,\mathrm ds .
\end{align*}
The nonnegative terminal variables on the left converge almost surely
to \(e^{c_{n,\beta}t}\rho_t^{(n)}\) as \(j\to\infty\).
Fatou's lemma therefore gives
\begin{equation}\label{finite-n-cost-bound}
 \mathbb E\rho_t^{(n)}
 \le e^{-c_{n,\beta}t}\rho_R(\xi,\xi')
   +\frac{C_R^{\rm err}}{c_{n,\beta}}
       (\delta+n^{-\beta}+n^{-1})
                    (1-e^{-c_{n,\beta}t}).
\end{equation}
Since \(n\ge2\), the rates satisfy
\begin{equation}\label{finite-n-rate-bounds}
 0<1\wedge\bigl[(1-2^{-\beta})\eta_{5,R}\bigr]
 \le c_{n,\beta}\le c,\qquad
 \lim_{n\to\infty}c_{n,\beta}=c.
\end{equation}

\medskip\noindent
\textbf{Step 2: the first path-space weak limit, \(n\to\infty\).}
Write \(\mathbf Q_{n,\delta}\) for the regularized path law from
Section~\ref{Sec3}, displaying the fixed switch width \(\delta\).
By Theorem~\ref{thwcc}, these laws are tight on
\(\Xi_2=C([0,\infty);E\times E)\) and have path marginals
\(\mathbf P^\xi,\mathbf P^{\xi'}\).
Choose a subsequence \(n_k\to\infty\) for which \(\mathbf Q_{n_k,\delta}\Longrightarrow\mathbf Q_\delta\) on \(\Xi_2\).
The continuous coordinate projections preserve the two path marginals, so
\(\mathbf Q_\delta\in\Pi(\mathbf P^\xi,\mathbf P^{\xi'})\).

For every fixed \(t\ge0\) and \(N>0\), the function \(\omega\mapsto
\rho_R\bigl((X_t,V_t)(\omega),(X'_t,V'_t)(\omega)\bigr)\wedge N\)
is bounded and continuous on \(\Xi_2\).
Indeed, evaluation at \(t\) is continuous in the locally uniform
topology, and \(\rho_R\) is continuous by \eqref{rho}. By weak convergence
\cite[Chapter~2, Definition~4.3, pp.~60--61]{KS1991}, its integral under
\(\mathbf Q_{n_k,\delta}\) converges to its integral under
\(\mathbf Q_\delta\). Each truncated integral is bounded by the
corresponding right-hand side of \eqref{finite-n-cost-bound}.
Using \eqref{finite-n-rate-bounds} and then letting \(N\uparrow\infty\)
by monotone convergence gives
\begin{equation}\label{first-path-limit-cost}
 \mathbb E_{\mathbf Q_\delta}
   \rho_R\bigl((X_t,V_t),(X'_t,V'_t)\bigr)
 \le e^{-ct}\rho_R(\xi,\xi')
        +\frac{C_R^{\rm err}\delta}{c}(1-e^{-ct}).
\end{equation}
The same subsequence of path laws works for every fixed time, since
weak convergence tests every bounded continuous function on \(\Xi_2\).
Thus this one law \(\mathbf Q_\delta\) satisfies
\eqref{first-path-limit-cost} for all \(t\ge0\).

\medskip\noindent
\textbf{Step 3: the second path-space weak limit, \(\delta\downarrow0\).}
Take \(\delta_j=1/(j+2)\).
Every law \(\mathbf Q_{\delta_j}\) in Step 2 has the same two path
marginals. Given \(\epsilon>0\), choose compact
\(\mathcal K_1,\mathcal K_2\subset\Xi_1\) with
\(\mathbf P^\xi(\mathcal K_1),\mathbf P^{\xi'}(\mathcal K_2)
\ge1-\epsilon/2\). Then
\[
 \mathbf Q_{\delta_j}\bigl((\mathcal K_1\times\mathcal K_2)^c\bigr)
 \le\mathbf P^\xi(\mathcal K_1^c)+\mathbf P^{\xi'}(\mathcal K_2^c)
 \le\epsilon,\qquad j\ge0.
\]
Prokhorov's theorem \cite[Chapter~2, Theorem~4.7, p.~62]{KS1991}
therefore gives a subsequence \(j_k\) and a path law \(\mathbf Q\)
such that \(\mathbf Q_{\delta_{j_k}}\Longrightarrow\mathbf Q\) on \(\Xi_2\).
The marginal projections are continuous, so
\(\mathbf Q\in\Pi(\mathbf P^\xi,\mathbf P^{\xi'})\).
Apply the same bounded-cost truncation argument as in Step 2 to
\eqref{first-path-limit-cost}. Since \(\delta_{j_k}\to0\), it gives
\begin{equation*}
 \mathbb E_{\mathbf Q}
   \rho_R\bigl((X_t,V_t),(X'_t,V'_t)\bigr)
 \le e^{-ct}\rho_R(\xi,\xi'),\qquad t\ge0.
\end{equation*}
Definition~\ref{wcoup} and Remark~\ref{rewc} now imply
\(\mathcal W_{\rho_R}(\delta_\xi P_t,\delta_{\xi'}P_t) \le e^{-ct}\rho_R(\xi,\xi').\)

\medskip\noindent
\textbf{Step 4: arbitrary initial laws with finite transport cost.}
Let \(\mu,\nu\in\mathcal P(E)\) satisfy
\(\mathcal W_{\rho_R}(\mu,\nu)<\infty\), and choose
\(\pi\in\Pi(\mu,\nu)\) with
\(\int_{E^2}\rho_R(\xi,\xi')\,\pi(\mathrm d\xi,\mathrm d\xi')<\infty\).
Using the construction in Theorem~\ref{thwcc}, start
\eqref{regu-coup} with joint initial law \(\pi\), independently
of the three Brownian motions.
Conditioning on the initial pair and integrating
\eqref{finite-n-cost-bound} against \(\pi\), with
\(C_R^{\rm err}\) independent of that pair, gives
\begin{equation}\label{arbitrary-initial-finite-n}
	\mathbb E_\pi\rho_t^{(n)}
	\le e^{-c_{n,\beta}t}
	\int_{E^2}\rho_R(\xi,\xi')\,\pi(\mathrm d\xi,\mathrm d\xi')
	+\frac{C_R^{\rm err}}{c_{n,\beta}}
	(\delta+n^{-\beta}+n^{-1})(1-e^{-c_{n,\beta}t}).
\end{equation}
No separate \(F\)-moment assumption on either marginal is needed.
The marginal path laws are
\[
\mathbf P^\mu:=\int_E\mathbf P^\xi\,\mu(\mathrm d\xi),
\qquad
\mathbf P^\nu:=\int_E\mathbf P^{\xi'}\,\nu(\mathrm d\xi'),
\]
independently of the regularization parameters.
Since these are fixed probability measures on the Polish space
\(\Xi_1\), the family of joint path laws is tight.
Repeating the limiting arguments in Steps 2 and 3, with
\eqref{arbitrary-initial-finite-n} in place of
\eqref{finite-n-cost-bound}, yields
\(\mathbf Q\in\Pi(\mathbf P^\mu,\mathbf P^\nu)\) satisfying
\[
\mathbb E_{\mathbf Q}
\rho_R\bigl((X_t,V_t),(X'_t,V'_t)\bigr)
\le e^{-ct}\int_{E^2}\rho_R\,\mathrm d\pi,
\qquad t\ge0.
\]
The time-\(t\) marginal of \(\mathbf Q\) belongs to
\(\Pi(\mu P_t,\nu P_t)\), so
\(\mathcal W_{\rho_R}(\mu P_t,\nu P_t) \le e^{-ct}\int_{E^2}\rho_R\,\mathrm d\pi.\)
Taking the infimum over all finite-cost
\(\pi\in\Pi(\mu,\nu)\) proves the claim when
\(\mathcal W_{\rho_R}(\mu,\nu)<\infty\).
If \(\mathcal W_{\rho_R}(\mu,\nu)=\infty\), the inequality
holds trivially in the extended-real sense.
\end{proof}

We next verify the moment bound needed for convergence to the Gibbs
probability, using the standard invariance property of Langevin dynamics.

\begin{proposition}\label{prop:gibbs}
Under Assumptions~\ref{as4.0}--\ref{as4.3}, the Gibbs probability
\(\mu_*\) satisfies
\[\mu_*P_t=\mu_*,\qquad t\ge0,\qquad \int_E F\,\mathrm d\mu_*<\infty.\]
\end{proposition}
\begin{proof}
By Assumption~\ref{as4.0}, \(U\) is \(C^1\), coercive, and has an integrable Gibbs density; Assumption~\ref{as4.1} makes its gradient locally Lipschitz, and Theorem~\ref{thwcc} gives the non-explosive dynamics.
Equation~\eqref{eq4.1} is the Langevin process in \cite[Example~3(ii), p.~180]{EberleLoerler2026}, with the same potential \(U\), friction \(\gamma\), and Gaussian velocity law \(N(0,\mathbf I_d)\).
The invariant product measure stated there is precisely \(\mu_*\), so \(\mu_*P_t=\mu_*\).
For the moment bound, completing the square gives
\[
\mathcal V(x,v)-U(x)-\frac12|v|^2
=-\frac34U(x)+\frac{7\gamma^2}{96}|x|^2
-\frac38\left|v-\frac\gamma6x\right|^2.
\]
Integrating in \(v\) and using \eqref{eq:quadratic-lower}, we obtain
\[
\begin{aligned}
\int_E F\,\mathrm d\mu_*
&=\frac{(8\pi/3)^{d/2}}{\mathfrak Z}
\int_{\mathbb R^d}
\exp\left(-\frac34U(x)+\frac{7\gamma^2}{96}|x|^2\right)\,\mathrm dx\\
&\le\frac{(8\pi/3)^{d/2}}{\mathfrak Z}
\exp\left(\frac{3L_2K_1^2}{8}\right)
\int_{\mathbb R^d}
\exp\left(-\frac{36L_2-7\gamma^2}{96}|x|^2\right)\,\mathrm dx
<\infty,
\end{aligned}
\]
where the last inequality follows from \(L_2>\gamma^2/4\).
\end{proof}

\begin{proof}[Proof of Corollary~\ref{Co1}]
Proposition~\ref{prop:gibbs} gives \(\mu_*P_t=\mu_*\) and
\(\mu_*F<\infty\). Taking \(\nu=\mu_*\) in Theorem~\ref{Th}
proves the stated exponential convergence.

To prove uniqueness, note that \(A_R\ge\alpha_R\ge1\) and
\(y=\gamma(q-\alpha_Rz)\) imply
\[
|\xi-\xi'|\le |z|+|y|
\le (1+\gamma\alpha_R)|z|+\gamma|q|
\le (\gamma+1)r(\xi,\xi').
\]
By the concavity and monotonicity of \(f_R\), with \(f_R(0)=0\),
and the lower bound one for the weight in \eqref{rho},
\(\rho_R(\xi,\xi')\ge f_R\bigl((\gamma+1)^{-1}\bigr)(|\xi-\xi'|\wedge1).\)
The coefficient is positive, so zero \(\rho_R\)-transportation cost
implies equality of probability measures.
If \(\nu\) is invariant and
\(\mathcal W_{\rho_R}(\nu,\mu_*)<\infty\), Theorem~\ref{Th} yields
\(\mathcal W_{\rho_R}(\nu,\mu_*) \le e^{-ct}\mathcal W_{\rho_R}(\nu,\mu_*),~t>0.\)
Since \(c>0\), this finite cost is zero, and hence \(\nu=\mu_*\).
\end{proof}

\section{Asymmetric quartic double wells and numerical contraction bounds}\label{Sec6}
We examine loss of dependence on the initial law for a quartic potential with unequal well depths. Tilted quartic double wells are used as test potentials for Langevin sampling \cite[Section~3]{LuSpiliopoulos2018}. Related profiles form the central part of trapping potentials used in colloidal relaxation experiments \cite[Methods]{KumarBechhoefer2020}. For a tilted variant of the symmetric wells \eqref{eq:quartic-double-well}, take $d=1$, $\gamma=1$, and
\begin{equation}\label{eq:asymmetric-quartic}
 U_\kappa^{\rm asym}(x)
 =\kappa\left(\frac{(x^2-1)^2}{4}+\frac{x}{5}+\frac3{10}\right),
 \qquad \kappa\ge1.
\end{equation}
Assumptions~\ref{as4.0}--\ref{as4.3} hold with
\[
 m=4,\quad \ell=2,\quad b_0=\frac\kappa8,\quad B_0=\frac{27\kappa}{100},
 \quad L_1=\frac{3\kappa}{2},\quad L_2=\frac12,\quad K_1=2,
 \quad L_3=1,\quad K_2=\frac{19\kappa}{30}.
\]
Indeed, writing $u=U_\kappa^{\rm asym}/\kappa$, the identities
\[
 xu'(x)-u(x)=\frac34\left(x^2-\frac13\right)^2-\frac{19}{30}
\]
give nonnegativity, the quartic lower bound, and the potential--dissipation inequality. The difference of the forces is bounded by $\kappa[1+3(x^2+y^2)/2]|x-y|$, and, for $|x|\ge2$, one has $x(U_\kappa^{\rm asym})'(x)\ge\kappa x^2(x^2-1-1/(5|x|))\ge x^2/2$. Finally, $3\ell=6<8=2m$, so Theorem~\ref{Th} applies.

The stationary points solve $x^3-x+1/5=0$. The two minima are $x_-\approx-1.08803391$ and $x_+\approx0.87888507$, separated by a maximum at $x_{\rm b}\approx0.20914885$. Since $u(-x)<u(x)$ for $x>0$, the left well is deeper. We use $\kappa=4$ for the main experiment and write $U=U_4^{\rm asym}$. The position Gibbs density and the full Gibbs density are
\begin{equation}\label{eq:example-gibbs}
 p_{*,X}(x)=\frac{e^{-U(x)}}{\mathfrak Z_{\rm pos}},\qquad
 p_*(x,v)=p_{*,X}(x)\frac{e^{-v^2/2}}{\sqrt{2\pi}}.
\end{equation}
Adaptive quadrature gives $\mathfrak Z_{\rm pos}\approx0.76425445$, $\mu_*(x)\approx-0.59379835$, and $\mu_*(x^2)\approx0.94081329$. The mass to the left of $x_{\rm b}$ is approximately $0.81815$, and the ratio of the two peak heights is approximately $4.91126$.

The primary experiment compares the laws from $\xi=(x_-,0)$ and $\xi'=(x_+,0)$, rather than only comparing a long trajectory with the Gibbs density. All computations use MATLAB R2024b. To retain finite exponential energy moments under quartic confinement, we use two exact Ornstein--Uhlenbeck half steps separated by a discrete-gradient Hamiltonian step \cite{Gonzalez1996}. For independent standard normal variables $Z_1,Z_2$, set $\eta=e^{-h/2}$ and $\sigma=\sqrt{1-e^{-h}}$, and update $(x,v)$ by
\begin{equation}\label{eq:numerical-update}
\begin{aligned}
 w&=\eta v+\sigma Z_1, &
 x^+&=x+\frac h2(w+w^+),\\
 w^+&=w-h\frac{U(x^+)-U(x)}{x^+-x}, &
 v^+&=\eta w^++\sigma Z_2.
\end{aligned}
\end{equation}
At $x^+=x$, the quotient means $U'(x)$. Its polynomial expression is used in the code. When $h^2\kappa/4<1$, the polynomial left-hand side of the implicit equation tends to $\pm\infty$ as $x^+\to\pm\infty$ and has derivative at least $1-h^2\kappa/4>0$, so the root exists and is unique. The Hamiltonian substep conserves $\mathscr H(x,v)=U(x)+v^2/2$: multiplying the position and velocity increments gives $U(x^+)-U(x)+((w^+)^2-w^2)/2=0$. The resulting numerical kernel, denoted by $K_h$, approximates the Langevin semigroup; it does not exactly preserve the Gibbs measure.

We construct a coupling of two copies of $K_h$. Away from close pairs, their Gaussian noises are reflected, with synchronous updates in the strip $|2(x-x')+v-v'|\le|x-x'|/2$. For close pairs, the common component $0.9\min\{p_h(\xi,\cdot),p_h(\xi',\cdot)\}$ of the two transition densities makes the updates coincide with positive probability; the residual components are sampled by rejection. Each residual trial has acceptance probability at least $0.1$. The density formula and the selection criterion are given in the accompanying MATLAB program. Both numerical marginals remain $K_h$, and the chains use identical updates after their first meeting. Write $\tau$ for this meeting index.

This construction gives an upper bound in the same cost $\rho_R$ as Theorem~\ref{Th}. To control its exponential weight, Young's inequality and the quartic lower bound give
\[
 \mathcal V(x,v)-\frac{\mathscr H(x,v)}2
 \le-\frac{U(x)}4+\frac{x^2}{8}
 \le\frac{27\kappa}{400}+\frac1{8\kappa}.
\]
Moreover, the Ornstein--Uhlenbeck generator satisfies $\mathcal L_{\rm OU}e^{\mathscr H/2}=(1/2-v^2/4)e^{\mathscr H/2}\le e^{\mathscr H/2}/2$. Energy conservation in the middle step of \eqref{eq:numerical-update} therefore yields, for a deterministic initial state,
\begin{equation}\label{eq:numerical-F-bound}
 \mathbb E_\xi F(\xi_j^h)
 \le\exp\left(\frac{27\kappa}{400}+\frac1{8\kappa}
              +\frac{jh}{2}+\frac{\mathscr H(\xi)}2\right).
\end{equation}
Since $f_R\le f_R(\eta_{4,R})$ and the coupling cost vanishes after meeting, \eqref{rho} and \eqref{eq:numerical-F-bound} imply
\begin{equation}\label{eq:numerical-cost-bound}
\begin{aligned}
 \frac{\mathcal W_{\rho_R}(\delta_\xi K_h^j,\delta_{\xi'}K_h^j)}
      {f_R(\eta_{4,R})}
 \le\mathbb P(\tau>j)+\varepsilon_R
 \exp\left(\frac{27\kappa}{400}+\frac1{8\kappa}+\frac{jh}{2}\right)
 \left(e^{\mathscr H(\xi)/2}+e^{\mathscr H(\xi')/2}\right).
\end{aligned}
\end{equation}
Thus the meeting experiment controls the theorem's weighted transportation cost for the numerical kernel, without identifying a coupling cost with its infimum. We also estimate the Euclidean coupling cost $\mathbb E|\xi_j^h-\xi_j^{\prime,h}|$, which bounds the joint $W_1$ distance from above.

The main run uses $24000$ independent pairs, $h=0.04$, $T=60$, and seed $20261021$. A step-halving run uses $h=0.02$, the same number of pairs, and seed $20261022$. A separate run with $12000$ pairs and seed $20261023$ checks Monte Carlo variation. Amplitudes $\kappa=1,8$ use $18000$ pairs each, $h=0.04$, and seeds $20261024,20261025$, respectively. Values are recorded every $0.4$ time units. The maximum relative Hamiltonian residual in the implicit steps is below $1.2\times10^{-15}$.

For a direct marginal-transport diagnostic, each ensemble is split into independent blocks of $1500$ pairs. Within each block, the first $750$ positions from one marginal and the last $750$ from the other are independent samples; their empirical $W_1$ distance is the mean absolute difference between their order statistics. Two disjoint groups of $750$ positions from the same marginal provide a finite-sample reference. The curves average the block statistics, and their approximate pointwise $95\%$ bands are $1.96$ block standard errors. These bands describe Monte Carlo variation of the empirical statistic; its positive finite-sample bias is indicated by the same-law reference. Meeting probabilities use pointwise Wilson intervals, and Euclidean coupling costs use standard errors across independent pairs.

Figure~\ref{fig:example-density} gives the potential and the terminal position histogram, using $75$ bins on $[-2.4,2.1]$. The histogram combines the two terminal marginals; uncertainty is computed across independent blocks, retaining the dependence within each coupled pair. The two steps resolve the unequal Gibbs peaks. This density comparison is a consistency check for the model and integrator, while Figure~\ref{fig:example-contraction} examines dependence on the initial law.

For $\kappa=4$, the empirical position distances at the two step sizes fall to the same-law sampling scale. The first recorded time below $10\%$ of the initial position distance is $7.6$ for $h=0.04$ and $7.6$ for $h=0.02$; the independent-seed check gives $7.6$. At $t=60$, the Euclidean coupling upper bounds are approximately $0.2314$ and $0.5471$, respectively. The normalized weighted-cost bounds in \eqref{eq:numerical-cost-bound} are approximately $0.1379$ and $0.3378$, with a weight correction smaller than $10^{-2\times10^{41}}$. These coupling bounds depend on the coupling and the step size; the stronger step-size agreement of the empirical marginal distances shows why their decay should be distinguished from the meeting efficiency of a particular numerical coupling.

To evaluate the quantitative content of the continuous-time estimate, we compute every constant in Section~\ref{Sec2} with the parameters above and $R=R_\star$. This calculation is performed in logarithms. Since $h_R'(s)\ge\gamma\eta_{1,R}$ and $h_R(0)=0$, we have $\Phi_R(s)\le1/(\gamma\eta_{1,R})$. Integrating $e^{h_R(s)}$ with the same derivative bound gives $J_{i,R}\le e^{h_R(t_i)}/(\gamma\eta_{1,R}^2)$, where $t_1=\eta_{2,R}$ and $t_2=\eta_{4,R}$. Since $\eta_{1,R}\ge1$, the second branch in each defining minimum is selected, so
\[
 \log\eta_{3,R}=\log(\gamma/32)-h_R(\eta_{2,R}),\qquad
 \log\eta_{5,R}=\log(\gamma/32)-h_R(\eta_{4,R}).
\]
For $\kappa=4$, one obtains $\log C_5\approx8280.32$, $\log_{10}R_\star\approx22.23$, and $-\log_{10}c\approx1.074\times10^{58}$. Theorem~\ref{Th} and $f_R\le f_R(\eta_{4,R})$ give the continuous-time reference bound
\[
 \frac{\mathcal W_{\rho_R}(\delta_\xi P_t,\delta_{\xi'}P_t)}{f_R(\eta_{4,R})}
 \le e^{-ct}\bigl(1+\varepsilon_RF(\xi)+\varepsilon_RF(\xi')\bigr).
\]
Its right-hand side is indistinguishable from one over the displayed time interval; it is a bound for $P_t$, not for $K_h$. The time for a tenfold reduction of this bound is $\log(10)/c$, which is far beyond the simulated time scale. Table~\ref{tab:quartic-rates} compares this guaranteed rate with the observed marginal-transport diagnostic.

\begin{table}[htbp]
\centering
\caption{Explicit constants at $R=R_\star$ and empirical position-transport diagnostics. The last column is the first recorded time at which the mean empirical position $W_1$ is below $10\%$ of its initial value, using $h=0.04$; it is not a certified mixing time.}
\label{tab:quartic-rates}
\begin{tabular}{rccc}
\hline
$\kappa$ & $\log_{10}R_\star$ & $-\log_{10}c$ & Observed time\\
\hline
1 & $19.62$ & $4.976\times10^{50}$ & $4.8$\\
4 & $22.23$ & $1.074\times10^{58}$ & $7.6$\\
8 & $23.60$ & $8.074\times10^{61}$ & $10.8$\\
\hline
\end{tabular}
\end{table}

These experiments examine contraction between two initial laws in an asymmetric quartic example covered by the improved growth condition. They also identify the limitation of the sufficient constants: the theorem gives a positive rate without amplitude smallness, but that rate is much smaller than the observed decay. The numerical bounds concern $K_h$, and the step-halving and independent-seed checks assess their sensitivity; they do not replace the continuous-time estimate of Theorem~\ref{Th}.

\begin{figure}[p]
\centering
\includegraphics[width=0.96\textwidth]{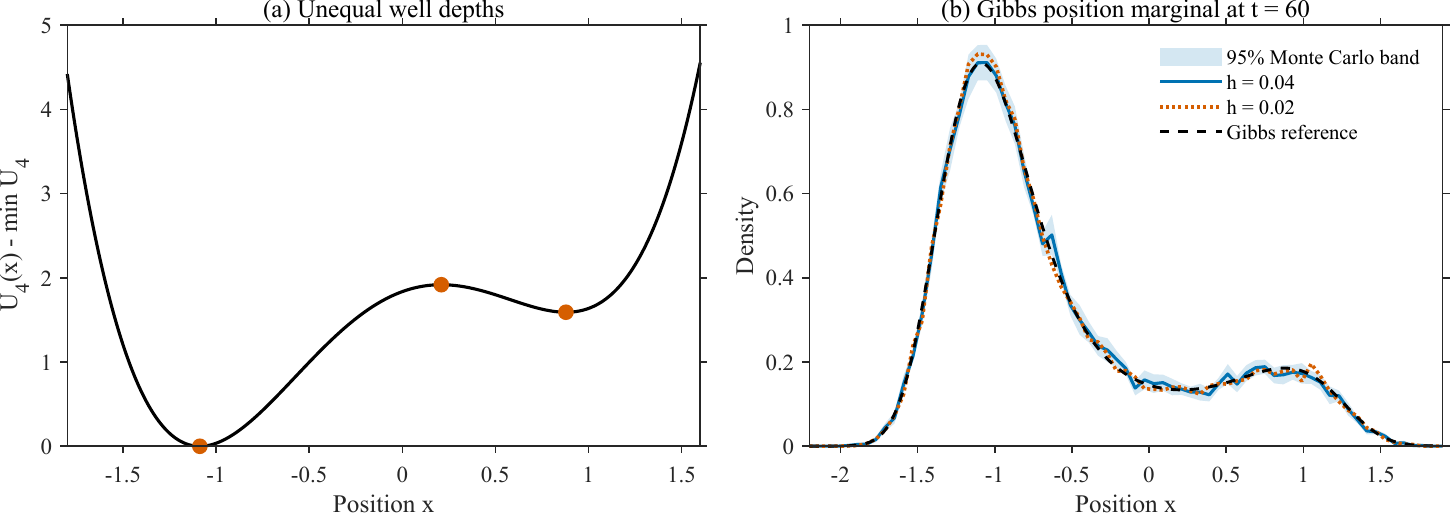}
\caption{Asymmetric quartic wells with $\kappa=4$. (a) The potential, with its minimum subtracted for display. (b) Terminal position histograms from $24000$ coupled pairs at $t=60$, using $h=0.04$ and $0.02$, against the Gibbs density \eqref{eq:example-gibbs}. The shaded band is an approximate pointwise $95\%$ Monte Carlo band for $h=0.04$, computed across independent blocks.}
\label{fig:example-density}
\medskip
\includegraphics[width=0.96\textwidth]{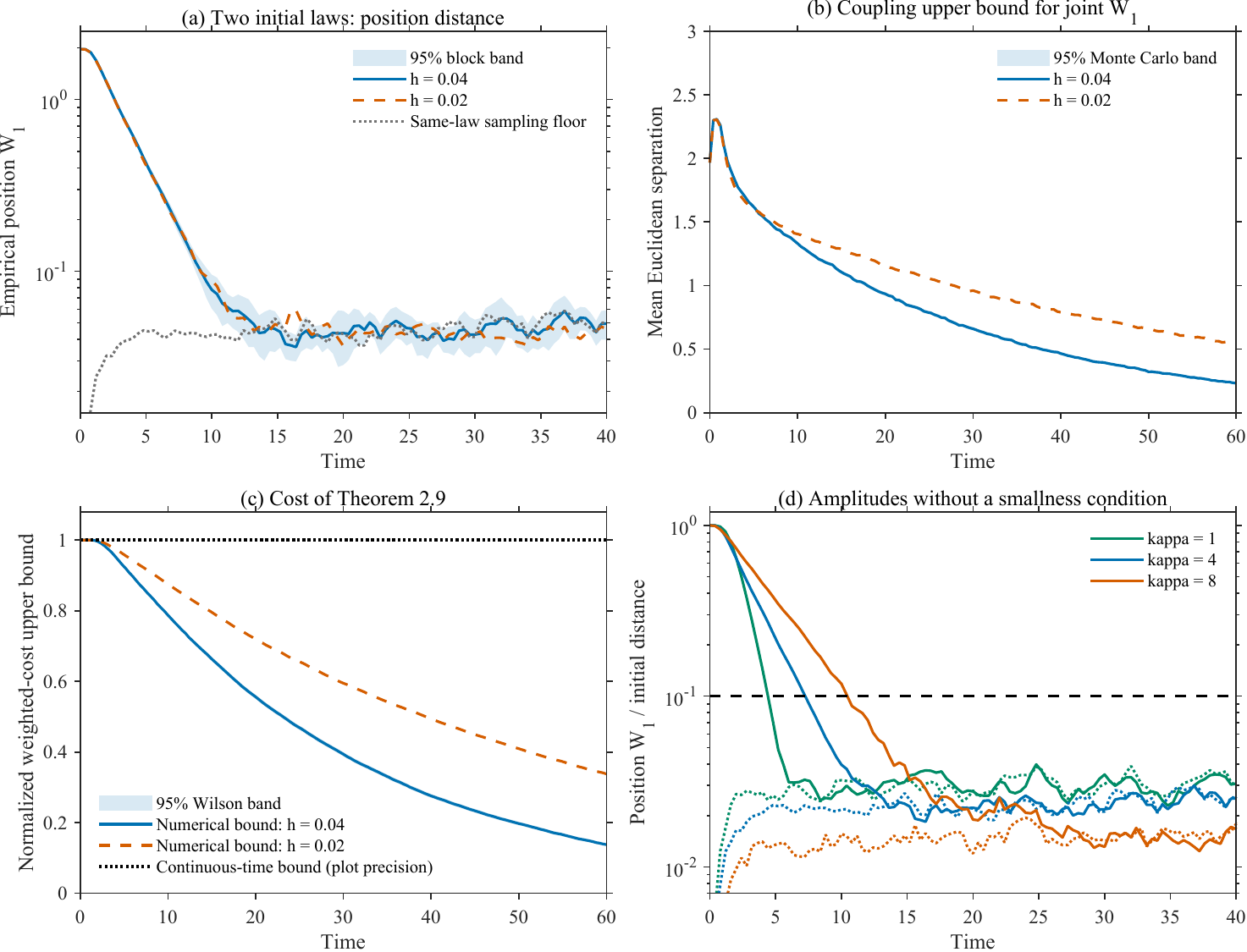}
\caption{Loss of dependence on initial states $(x_-,0)$ and $(x_+,0)$. (a) Block-averaged empirical position $W_1$, with a same-law sampling reference. (b) Euclidean coupling upper bounds for the joint $W_1$. (c) Upper bounds in the cost of Theorem~\ref{Th}, divided by $f_R(\eta_{4,R})$, from \eqref{eq:numerical-cost-bound}; the negligible weight correction is below plotting precision. The continuous-time bound is also shown at plotting precision. (d) Normalized empirical position distances for $\kappa=1,4,8$; dotted lines give the corresponding same-law references. Shaded bands are pointwise Monte Carlo intervals, as specified in the text.}
\label{fig:example-contraction}
\end{figure}
\FloatBarrier

\end{document}